\documentclass{article}

\usepackage[a4paper,left=3cm,right=3cm,top=3cm,bottom=4cm]{geometry}
\usepackage{amsmath,amssymb,amsfonts,yhmath}
\usepackage{mathrsfs}
\usepackage{hyperref}
\usepackage{graphicx}
\usepackage{epstopdf}
\usepackage{multirow}
\usepackage{xcolor}
\usepackage{enumitem}
\setlist[enumerate]{leftmargin=.5in}
\setlist[itemize]{leftmargin=.5in}

\graphicspath{{images/}{./}}

\definecolor{frenchblue}{rgb}{0.0, 0.45, 0.73}
\usepackage{algpseudocode}
\algnewcommand{\LineComment}[1]{\State {\color{frenchblue}\(\triangleright\) #1}}
\usepackage{algorithm}

\usepackage{tikz}
\usetikzlibrary{shapes,arrows}
\usetikzlibrary{calc,patterns,angles,quotes}
\usetikzlibrary{matrix}
\usepackage{tkz-euclide}
\usetikzlibrary{tikzmark}

\DeclareMathOperator{\diag}{diag}
\def\0{\mathbf{0}}
\def\1{\mathbf{1}}
\def\b{\mathbf{b}}

\def\f{\mathbf{f}}
\def\g{\mathbf{g}}
\def\h{\mathbf{h}}
\def\p{\mathbf{p}}

\def\r{\mathbf{r}}

\def\v{\mathbf{v}}

\def\x{\mathbf{x}}
\def\y{\mathbf{y}}
\def\z{\mathbf{z}}

\def\R{\mathbb{R}}
\def\A{\mathcal{A}}
\def\E{\mathcal{E}}
\def\F{\mathcal{F}}

\def\M{\mathcal{M}}

\def\S{\mathcal{S}}
\def\T{\mathcal{T}}
\def\V{\mathcal{V}}
\def\I{\mathtt{I}}
\def\B{\mathtt{B}}

\def\d{\mathrm{d}}
\def\D{\mathbb{D}}
\def\btheta{\boldsymbol{\theta}}

\newcommand{\argmin}{\operatornamewithlimits{argmin}}

\usepackage{amsthm}
\newtheorem{theorem}{Theorem}
\newtheorem{corollary}{Corollary}
\newtheorem{lemma}{Lemma}

\title{Convergent Authalic Energy Minimization for Disk Area-Preserving Parameterizations}

\author{
  Shu-Yung Liu\\
  \href{mailto:lii227857@gmail.com}{lii227857@gmail.com}
  \and
  Mei-Heng Yueh\\
  \href{mailto:yue@ntnu.edu.tw}{yue@ntnu.edu.tw}
}

\date{Department of Mathematics\\ National Taiwan Normal University}

\begin{document}
\begin{sloppypar}

\maketitle

\begin{abstract}
An area-preserving parameterization is a bijective mapping that maps a surface onto a specified domain and preserves the local area. This paper formulates the computation of disk area-preserving parameterization as an authalic energy minimization (AEM) problem and proposes a novel preconditioned nonlinear conjugate gradient method for the AEM with guaranteed theoretical convergence. Numerical experiments indicate that our new approach has significantly improved area-preserving accuracy and computational efficiency compared to another state-of-the-art algorithm. Furthermore, we present an application of surface registration to illustrate the practical utility of area-preserving mappings as parameterizations of surfaces.
\end{abstract}

\section{Introduction}

An area-preserving parameterization, commonly called an authalic or equiareal parameterization, is a mapping on a surface that preserves the local area up to a global scaling factor. The factor is precisely the ratio between the total area of the surface and that of the specified target domain.
Classical and state-of-the-art methods for area-preserving mapping include the stretch-minimizing method \cite{SaSG01,YoBS04}, Lie advection method \cite{ZoHG11}, optimal mass transportation method \cite{DoTa10,ZhSG13,SuCQ16}, diffusion-based method \cite{ChRy18}, and the stretch energy minimization (SEM) \cite{YuLW19,Yueh22}. 
Area-preserving mappings could serve as parameterizations and applied in tasks such as surface resampling, remeshing, and registration \cite{FlHo05,ShPR06,HoLe07,LaLu14,LuLY14,YoMY14,YuLW17}. They can also serve as boundary parameterizations in image preprocessing for brain tumor segmentation \cite{YuLL19,YuLL20,LiJY21,LiLH22}.
In particular, it has been demonstrated that the fixed-point iteration of the SEM \cite{YuLW19} for computing area-preserving mapping is more efficient and accurate than other state-of-the-art algorithms \cite{YoBS04,ZhSG13}. 
Moreover, a recent study has established the equivalence between stretch energy minimizers and area-preserving mappings \cite{Yueh22}, which further reinforces the theoretical support of SEM. 
Specifically, under the assumption that the total area of the domain and the image of mappings are the same, the minimal value of the stretch energy is the value of the image area, and the equality holds if and only if the mapping is area-preserving \cite{Yueh22}. 
Given this fact, it is natural to define the authalic energy as the difference between the stretch energy and the image area. This ensures that a mapping is area-preserving if and only if it has zero authalic energy.
Related previous works either consider the case that the boundary mapping is prescribed \cite{Yueh22} or assume the image area is constant \cite{YuLW19} so that it suffices to minimize the stretch energy when computing area-preserving parameterizations.
However, when the boundary vertices are allowed to glide along the unit circle, the exact image shape is a polygon with vertices on the unit circle so that the precise area of the image also depends on the boundary mapping, not the constant $\pi$. 
Due to this, we develop a modified authalic energy minimization (AEM) algorithm for computing disk area-preserving parameterizations.

\subsection{Contribution}

In this paper, we propose an efficient algorithm of the modified AEM to compute area-preserving parameterizations of simply connected open surfaces. The modification of the objective energy functional aims to remove the assumption of constant total image area. The contributions of this paper are as follows:
\begin{itemize}
\item[(i)] We formulate the disk area-preserving parameterization problem as a modified AEM problem and develop an associated preconditioned nonlinear conjugate gradient (CG) method for solving the problem. 
\item[(ii)] We remove the constant total image area assumption of \cite[Theorem 3.3]{Yueh22} and provide a generalized consequence that the mappings with the theoretical minimal stretch energy are area-preserving and vice versa.
\item[(iii)] We rigorously prove the convergence of the proposed method of the modified AEM for computing area-preserving parameterizations.
\item[(iv)] Numerical comparisons with another state-of-the-art algorithm are demonstrated to validate the significantly improved effectiveness of our new method in terms of area-preserving accuracy and computational efficiency.
\item[(v)] Application to area-preserving registration mappings between surfaces is demonstrated to show the utility of the proposed algorithm.
\end{itemize}
To the best of our knowledge, this is the first work that proposes a convergent algorithm that minimizes the authalic energy without the constant total image area assumption.

\subsection{Notation}
In this paper, we use the following notation.
\begin{itemize}[label={$\bullet$}]
\item Real-valued vectors are represented using bold letters, such as $\mathbf{f}$.
\item Real-valued matrices are represented using capital letters, such as $L$.
\item $I_n$ denotes the identity matrix of size $n\times n$.
\item Ordered sets of indices are denoted using typewriter letters, such as $\mathtt{I}$ and $\mathtt{B}$.
\item The $i$th entry of the vector $\mathbf{f}$ is denoted by $\mathbf{f}_i$.
\item The subvector of $\mathbf{f}$, composed of $\mathbf{f}_i$ for $i$ in the set of indices $\mathtt{I}$, is represented as $\mathbf{f}_\mathtt{I}$.
\item The $(i,j)$th entry of the matrix $L$ is represented by $L_{i,j}$.
\item The submatrix of $L$, composed of entries $L_{i,j}$ for $i$ in the set of indices $\mathtt{I}$ and $j$ in the set of indices $\mathtt{J}$, is denoted by $L_{\mathtt{I},\mathtt{J}}$.
\item The set of real numbers is represented by $\mathbb{R}$.
\item The $k$-simplex with vertices ${v}_0, \ldots, {v}_k$ is denoted by $\left[{v}_0, \ldots, {v}_k\right]$.
\item The area of a $2$-simplex $\left[{v}_0, v_1, {v}_2\right]$ is represented by $|\left[{v}_0, v_1, {v}_2\right]|$.
\item The area of a simplicial $2$-complex $\S$ is denoted by $|\S|$.
\item The zero and one vectors or matrices of appropriate sizes are denoted by $\mathbf{0}$ and $\mathbf{1}$, respectively.
\end{itemize}

\subsection{Organization of the paper}

The remaining part of this paper is organized as follows: 
In section \ref{sec:2}, we introduce the simplicial surfaces, mappings, and the stretch energy functional. Then, we propose the nonlinear CG method of the AEM for computing area-preserving parameterizations of surfaces in section \ref{sec:3}. The convergence analysis of the proposed iterative method is provided in section \ref{sec:4}. Numerical experiments and a comparison to another state-of-the-art algorithm are demonstrated in section \ref{sec:5}. An application of the AEM algorithm to the computation of area-preserving registration mappings is presented in section \ref{sec:6}. Concluding remarks are provided in section \ref{sec:7}.

\section{Authalic energy of simplicial mappings}
\label{sec:2}

A simplicial surface $\S \subset \mathbb{R}^3$ is the underlying set of a simplicial 2-complex consisted of $n$ vertices
\begin{subequations} \label{eq:mesh}
\begin{equation}
\mathcal{V}(\mathcal{S}) = \left\{ v_\ell = (v_\ell^1, v_\ell^2, v_\ell^3) \in\mathbb{R}^3 \right\}_{\ell=1}^n,
\end{equation}
and $m$ oriented triangular faces 
\begin{equation}
\mathcal{F}(\S) = \left\{ \tau_s = [ {v}_{i_s}, {v}_{j_s}, {v}_{k_s} ] \subset\R^3 \mid {v}_{i_s}, {v}_{j_s}, {v}_{k_s}\in\mathcal{V}(\S) \right\}_{s=1}^m,
\end{equation}
where the bracket $[{v}_{i_s}, {v}_{j_s}, {v}_{k_s}]$ denotes a $2$-simplex, i.e., a triangle with vertices being ${v}_{i_s}, {v}_{j_s}, {v}_{k_s}$. 
The set of edges is denoted as
\begin{equation}
\E(\S) = \left\{ [v_i,v_j] \subset\R^3 \mid [v_i,v_j,v_k]\in\mathcal{F}(\S) \text{ for some $v_k\in\V(\S)$} \right\}.
\end{equation}
\end{subequations}

A simplicial mapping $f:\S\to\R^2$ on $\S$ is a mapping such that $f(\S)$ is the underlying set of the simplicial $2$-complex composed of vertices $\V(f(\S))$, edges $\E(f(\S))$, and triangular faces $\F(f(\S))$ as in \eqref{eq:mesh}. 
In other words, $f$ is a piecewise affine mapping that satisfies
$$
f(v_i) = (f_i^1, f_i^2)^\top, ~ \text{ for every $v_i\in\V(\S)$,}
$$
and
$$
f|_{\tau_s}(v) =  \frac{1}{|\tau_s|} \Big(|[v,v_{j_s},v_{k_s}]| \, f(v_{i_s}) + |[v_{i_s},v,v_{k_s}]| \, f(v_{j_s}) + |[v_{i_s},v_{j_s},v]| \, f(v_{k_s}) \Big),
$$
for every $\tau_s\in\F(\S)$, where $|\tau_s|$ denote the area of the triangle $\tau_s$. 
As a result, the simplicial mapping $f$ can be represented as a vector
$$
\f \equiv \big( (\f^1)^\top, (\f^2)^\top \big)^\top \equiv (f_1^1, \ldots, f_n^1, f_1^2, \ldots, f_n^2)^\top \in\R^{2n}.
$$

The authalic energy for the simplicial mapping $f:\M\to\R^2$ is defined as
\begin{subequations} \label{eq:E_S}
\begin{equation} \label{eq:E_S1}
{E}_A({f}) = E_S(f) - \A(f),
\end{equation}
where $E_S$ is the stretch energy functional defined as
\begin{equation} \label{eq:E_S0}
E_S(f) = \frac{1}{2}\f^\top \big(I_2\otimes L_S({f})\big) \,\f
\end{equation}
with $L_S({f})$ being the weighted Laplacian matrix given by
\begin{equation} \label{eq:L_S}
{[L_S(f)]}_{i,j} =
   \begin{cases}
   -\sum_{[v_i,v_j,v_k]\in\F(\S)} [\omega_S(f)]_{i,j,k}  &\mbox{if $[{v}_i,{v}_j]\in\mathcal{E}(\S)$,}\\[0.1cm]
   -\sum_{\ell\neq i} [L_S(f)]_{i,\ell} &\mbox{if $j = i$,}\\[0.1cm]
   0 &\mbox{otherwise,}
   \end{cases}
\end{equation}
$\omega_S$ is the modified cotangent weight defined as
\begin{equation} \label{eq:omega}
[\omega_S(f)]_{i,j,k} = \frac{\cot(\theta_{i,j}^k(f))  \, |f([v_i,v_j,v_k])|}{2|[v_i,v_j,v_k]|}
\end{equation}
in which $\theta_{i,j}^k(f)$ is the angle opposite to the edge $[f(v_i),f(v_j)]$ at the point $f(v_k)$ on the image $f(\S)$, as illustrated in Figure \ref{fig:cot}, and $\A$ is the area measurement of the image of $f$ given by $\A(f)=|f(\S)|$.  
\end{subequations}
\begin{figure}
\centering
\begin{tikzpicture}[thick,scale=1.2]
\coordinate (v_i) at (0,0);
\coordinate (v_j) at (0,2);
\coordinate (v_k) at (2,1);
\coordinate (v_l) at (-2,1);
\filldraw[yellow!20] (v_i) -- (v_j) -- (v_k);
\filldraw[yellow!20] (v_i) -- (v_j) -- (v_l);
\pic[draw, ->, "$\theta_{i,j}^k(f)$", angle eccentricity=2.05, angle radius=0.6cm]{angle = v_i--v_l--v_j};
\pic[draw, ->, "$\theta_{j,i}^\ell(f)$", angle eccentricity=2.05, angle radius=0.6cm]{angle = v_j--v_k--v_i};
\draw{
(v_i) -- (v_j) -- (v_k) -- (v_i) -- (v_l) -- (v_j)
};
\tikzstyle{every node}=[circle, draw, fill=cyan!20,
                        inner sep=1pt, minimum width=2pt]
\draw{
(0,0) node{$f_i$}
(0,2) node{$f_j$}
(2,1) node{$f_\ell$}
(-2,1) node{$f_k$}
};
\end{tikzpicture}
\label{fig:cot}
\caption{An illustration of the cotangent weight defined on the image of $f$. Here, $f_i$, $f_j$, $f_k$ and $f_\ell$ denote $f(v_i)$, $f(v_j)$, $f(v_k)$ and $f(v_\ell)$, respectively.}
\end{figure}
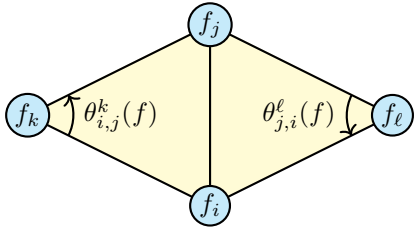

Under the assumption that the total area remains unchanged, i.e., $\A(f) = |\S|$, the authalic energy $E_A(f) \geq 0$, and the equality holds if and only if $f$ is area-preserving \cite[Corollary 3.4]{Yueh22}. 
When the target domain is a unit disk with boundary vertices mapped to the unit circle, the exact image $f(\S)$ is a polygon that depends on $f$. Consequently, the condition $\A(f) = |\S|$ would not always hold. 
To remedy this issue, we define the normalized authalic energy functional as
\begin{subequations} \label{eq:E_A_mod}
\begin{equation} \label{eq:E_A_moda}
\widetilde{E}_A({f}) = \frac{1}{2} \f^\top \big(I_2\otimes \widetilde{L_S}({f})\big) \,\f - \A(f),
\end{equation}
where $\widetilde{L_S}({f})$ being the weighted Laplacian matrix given by
\begin{equation} \label{eq:L_S_mod}
{[\widetilde{L_S}(f)]}_{i,j} =
   \begin{cases}
   -\sum_{[v_i,v_j,v_k]\in\F(\S)} [\widetilde{\omega_S}(f)]_{i,j,k}  &\mbox{if $[{v}_i,{v}_j]\in\mathcal{E}(\S)$,}\\[0.1cm]
   -\sum_{\ell\neq i} [\widetilde{L_S}(f)]_{i,\ell} &\mbox{if $j = i$,}\\[0.1cm]
   0 &\mbox{otherwise,}
   \end{cases}
\end{equation}
$\widetilde{\omega_S}$ is the normalized weight defined as
\begin{equation} \label{eq:omega_mod}
[\widetilde{\omega_S}(f)]_{i,j,k} = \frac{\cot(\theta_{i,j}^k(f))  \, |f([v_i,v_j,v_k])|}{2|[v_i,v_j,v_k]| \, \frac{|f(\S)|}{|\S|}} = \frac{|\S|}{|f(\S)|}[\omega_S(f)]_{i,j,k}.
\end{equation}
\end{subequations}
The scaling factor to $|\tau|$ in \eqref{eq:omega_mod} enforces $\sum_{\tau\in\F(\S)}|\tau| \frac{|f(\S)|}{|\S|} = |f(\S)|$ so that the assumption of \cite[Corollary 3.4]{Yueh22} would always hold. Equivalently, 
$\widetilde{L_S}({f}) = \frac{|\S|}{|f(\S)|} {L_S}({f})$, and thus,
\begin{equation} \label{eq:tilde_E}
\widetilde{E}_A({f}) = \frac{|\S|}{|f(\S)|} E_S(f) - \A(f).
\end{equation}

To guarantee area-preserving mappings minimize the modified energy functional \eqref{eq:tilde_E}, we prove the following theorem, which is a generalization of \cite[Theorem 3.3]{Yueh22} without assuming the constant total image area.

\begin{theorem}
The modified authalic energy \eqref{eq:E_A_mod} satisfies $\widetilde{E}_A({f}) \geq 0$, and the equality holds if and only if $f$ is area-preserving.
\end{theorem}
\begin{proof}
From \cite[Lemma 3.1]{Yueh22}, we know that the stretch energy \eqref{eq:E_S0} can be reformulated as
$$
E_S(f) = \sum_{\tau\in\F(\S)} \frac{|f(\tau)|^2}{|\tau|}.
$$
By the Cauchy--Schwarz inequality, we obtain
\begin{align*}
E_S(f) \, |\S| &= \sum_{\tau\in\F(\S)} \frac{|f(\tau)|^2}{|\tau|} 
\sum_{\tau\in\F(\S)} |\tau|
\geq \left(\sum_{\tau\in\F(\S)} |f(\tau)| \right)^2 = |f(\S)|^2. 
\end{align*}
That is, 
$$
\frac{|\S|}{|f(\S)|}E_S(f) \geq |f(\S)| = \A(f).
$$
The equality holds if and only if $\frac{|f(\tau)|^2}{|\tau|^2}$ is constant, i.e., $f$ preserves the area.
\end{proof}

In particular, suppose $\S$ is a simply connected open simplicial surface with a boundary $\partial\S$. We denote the index sets of entries of $\f$ correspond to boundary and interior vertices as 
\begin{equation} \label{eq:BI}
\B = \{ b \mid v_b \in \partial\S \} ~ \text{ and } ~ \I = \{1, \ldots, n\}\backslash\B,
\end{equation}
respectively. 
Noting that each entry of $\f_\B$ is located on the unit circle, we can represent the boundary map using polar coordinates as
\begin{equation}
\begin{bmatrix}
\f_\B^1 & \f_\B^2
\end{bmatrix}
=\begin{bmatrix}
\cos\theta_1 & \sin\theta_1 \\
\cos\theta_2 & \sin\theta_2 \\
\vdots & \vdots \\
\cos\theta_{n_\B} & \sin\theta_{n_\B}
\end{bmatrix}
:= \begin{bmatrix}
\cos\btheta & \sin\btheta
\end{bmatrix},
\end{equation}
where $n_\B$ is the number of boundary vertices, $\btheta = (\theta_1, \theta_2, \ldots, \theta_{n_\B})^\top$ with $0=\theta_1 \leq \theta_2 \leq \cdots \leq \theta_{n_\B} \leq 2\pi$. 
Then, the simplicial map $f$ can be represented as a vector
\begin{equation} \label{eq:vecf}
\widehat{\f} = 
\begin{bmatrix}
\f_\I^1 \\
\f_\I^2 \\
\btheta
\end{bmatrix}.
\end{equation}

The image area $\A(f)$ is determined by the boundary map $\begin{bmatrix}
\f_\B^1 & \f_\B^2
\end{bmatrix}$ formulated as
\begin{align}
\A(f) &= \frac{1}{2} \sum_{i=1}^{n_\B} \sin (\theta_{i+1} - \theta_i) \quad\quad (\theta_{n_\B+1}:=\theta_1) \nonumber\\
&= \frac{1}{2}
\begin{bmatrix}
\cos \theta_1\\
\cos \theta_2\\
\vdots\\
\cos \theta_{n_\B}
\end{bmatrix}^{\top}
\begin{bmatrix}
 0 &  1     &        & -1 \\
-1 &  0     & \ddots &    \\
   & \ddots & \ddots &  1 \\
 1 &        & -1     &  0
\end{bmatrix}
\begin{bmatrix}
\sin \theta_1\\
\sin \theta_2\\
\vdots\\
\sin \theta_{n_\B}
\end{bmatrix} := \frac{1}{2} {\f_\B^1}^{\top} D\, \f_\B^2. \label{eq:PolarArea}
\end{align}
Then, by substituting \eqref{eq:PolarArea} into \eqref{eq:tilde_E}, we obtain a reformulation of the normalized authalic energy functional as
\begin{align} \label{eq:E_A_1}
\widetilde{E}_A(\f_\I^1, \f_\I^2, \btheta) 
&= \frac{2|\S|}{{\f_\B^1}^{\top} D\, \f_\B^2} E_S(f) - \frac{1}{2} {\f_\B^1}^{\top} D\, \f_\B^2.
\end{align}
Ultimately, the AEM problem can be explicitly formulated as
\begin{equation} \label{eq:AEM}
\argmin \widetilde{E}_A(\f_\I^1, \f_\I^2, \btheta) ~ \text{ subject to \, $0=\theta_1\leq\theta_2\leq\cdots\leq\theta_{n_\B}\leq 2\pi$}.
\end{equation}

It is known that $\nabla E_S$ and $\nabla_{\btheta} \A(f)$ can be written in neat formulations as
\begin{equation} \label{eq:GradE_S}
\nabla E_S(f) = 2 (I_2\otimes L_S(f)) \, \f,
\end{equation}
and
\begin{equation} \label{eq:GradA}
\nabla_{\btheta} \mathcal{A}(f) = \frac{1}{2} \nabla_{\btheta} {\f_\B^1}^{\top} D\, \f_\B^2 = -\frac{1}{2} ( \diag(\f_\B^1) D\, \f_\B^1 + \diag(\f_\B^2) D\, \f_\B^2), 
\end{equation}
respectively. Detailed derivations can be found in \cite{Yueh22,KuLY21}. 
Then, by the chain rule together with \eqref{eq:GradE_S} and \eqref{eq:GradA}, the gradient of $\widetilde{E}_A$ in \eqref{eq:E_A_1} can be explicitly formulated as
\begin{subequations} \label{eq:GradEA}
\begin{equation}
\nabla \widetilde{E}(\f_{\I}^1, \f_{\I}^2, \btheta) 
\equiv \nabla \widetilde{E}_A(\widehat{\f}) 
= \begin{bmatrix}
\nabla_{\f_{\I}^1} \widetilde{E}_A(\f_{\I}^1, \f_{\I}^2, \btheta) \\
\nabla_{\f_{\I}^2} \widetilde{E}_A(\f_{\I}^1, \f_{\I}^2, \btheta) \\
\nabla_{\btheta} \widetilde{E}_A(\f_{\I}^1, \f_{\I}^2, \btheta)
\end{bmatrix}
\end{equation}
with
\begin{align}
\nabla_{\f_{\I}^s} \widetilde{E}_A(\f_{\I}^1, \f_{\I}^2, \btheta) &= \Big( \frac{2|\S|}{{\f_\B^1}^{\top} D\, \f_\B^2} \Big) \nabla_{\f_{\I}^s} E_S(f) \nonumber \\
&= \Big( \frac{4 |\S|}{{\f_\B^1}^{\top} D\, \f_\B^2} \Big) \,( [L_S(f)]_{\I, \I} \f_{\I}^s + [L_S(f)]_{\I,\B} \f_\B^s ),
\end{align}
for $s=1,2$, and
{\small\begin{align}
\nabla_{\btheta} \widetilde{E}_A(\f_{\I}^1, \f_{\I}^2, \btheta) 
&= \frac{2|\S|}{{\f_\B^1}^{\top} D\, \f_\B^2} 
\nabla_{\btheta} E_S(f) + E_S(f) \nabla_{\btheta} \Big( \frac{2|\S|}{{\f_\B^1}^{\top} D\, \f_\B^2} \Big) - \nabla_{\btheta} \mathcal{A}(f) \nonumber \\
&= \frac{4 |\S|}{{\f_\B^1}^{\top} D\, \f_\B^2}\, \big( \diag(\f_\B^1)( [L_S(f)]_{\B,\I} \f^2_{\I} + [L_S(f)]_{\B, \B} \f_\B^2) \nonumber \\
& \quad\quad\quad\quad~~ - \diag(\f_\B^2)( [L_S(f)]_{\B,\I} \f^1_{\I} + [L_S(f)]_{\B,\B} \f_\B^1) \big) \nonumber \\
&\quad + \left(\frac{1}{2} + \frac{2 |\S| \, E_S(f)}{({\f_\B^1}^{\top} D\, \f_\B^2)^2} \right) (\diag (\f_\B^1) D\, \f_\B^1 + \diag(\f_\B^2) D\, \f_\B^2).
\end{align}}
\end{subequations}
Based on explicit formulae in \eqref{eq:GradEA}, we are able to develop the nonlinear CG method for solving \eqref{eq:AEM}, which will be described in detail in the next section.

\section{Nonlinear CG Method for AEM}
\label{sec:3}

In the previous work \cite{YuLW19}, the image area is assumed to be constant. In that case, minimizing $E_A$ in \eqref{eq:E_S1} is equivalent to minimizing $E_S$ in \eqref{eq:E_S0}.
To find a minimizer of $E_S$, the fixed-point iteration with a boundary update technique was adopted in \cite{YuLW19}, which is formulated as
\begin{subequations} \label{eq:SEM}
\begin{align}
{\f_\I^s}^{(k+1)} &= -[L_S(f^{(k)})]_{\I,\I}^{-1} [L_S(f^{(k)})]_{\I,\B} {\f_\B^s}^{(k)}, \label{eq:SEM_iter} \\[0.1cm]
{\f_\B^s}^{(k+1)} &= -N^{(k+1)} C [L_S(f^{(k)})]_{\B,\B}^{-1} [L_S(f^{(k)})]_{\B,\I} R^{(k+1)} {\f_\I^s}^{(k+1)}, 
\end{align}
where $R^{(k+1)}$ is the inversion operator for $f^{(k+1)}$ given by
\begin{equation}
R^{(k+1)} = \diag\left(\big\|\big( (\f_{\I}^1)_\ell^{(k+1)}, (\f_{\I}^2)_\ell^{(k+1)} \big)\big\|^{-2}\right)_{\ell=1}^{n_\I},
\end{equation}
$n_\I$ is the number of interior vertices, $C$ is the centralization matrix given by
\begin{equation}
C = I_{n_\B} - \frac{1}{n_\B}(\1_{n_\B} \1_{n_\B}^\top),
\end{equation}
$N^{(k+1)}$ is the normalization operator for $\f_\I^{(k+1)}$ given by
\begin{equation}
N^{(k+1)} = \diag\left( \big\| \big[C [L_S(f^{(k)})]_{\B,\B}^{-1} [L_S(f^{(k)})]_{\B,\I} R^{(k+1)} \f_\I^{(k+1)} \big]_\ell \big\|^{-1} \right)_{\ell=1}^{n_\I},
\end{equation}
\end{subequations}
$L_S(f^{(0)}) = L_S(\mathrm{id})$ with $\mathrm{id}$ being the identity map on $\S$, 
and $f|_{\partial\S}^{(0)}$ is an arc-length boundary parameterization.

Numerical experiments in \cite{YuLW19} indicate that the fixed-point iteration \eqref{eq:SEM} of SEM is effective on decreasing the stretch energy \eqref{eq:E_S0}. However, the convergence of \eqref{eq:SEM} is not theoretically guaranteed. 
In addition, the boundary mapping is not guaranteed to be updated in the optimal direction. 
To remedy these drawbacks, we are motivated to develop the nonlinear CG method for the AEM problem \eqref{eq:AEM}.

The concept of the nonlinear CG method for solving nonlinear optimization problems was first introduced by Fletcher and Reeves \cite{FlRe64}. Given the availability of the explicit formula \eqref{eq:GradEA} for the objective nonlinear functional \eqref{eq:E_A_1}, this concept aligns well with our intended objective.

Suppose an initial map $f^{(0)}$ is computed by the fixed-point iteration \eqref{eq:SEM_iter} of SEM with the arc-length parameterization as the fixed boundary map and $N$ iteration steps, e.g., $N=5$ in practice. 
Let $\g^{(0)} = \nabla \widetilde{E}_A(\widehat{\f}^{(0)})$ be the gradient vector defined as \eqref{eq:GradEA} and $\widehat{\f}^{(0)}$ be the vector representation of $f^{(0)}$ as in \eqref{eq:vecf}. 
The map is updated along with the negative gradient direction as
\begin{subequations}\label{eq:CG_direction}
\begin{equation}
\widehat{\f}^{(1)} = \widehat{\f}^{(0)} + \alpha_0 {\p}^{(0)},
\end{equation}
where 
\begin{equation} \label{eq:CG_direction0}
{\p}^{(0)} = -{\g}^{(0)},
\end{equation}
and $\alpha_0$ is a suitable step size. 
The choice of step sizes will be discussed in subsection \ref{subsec:3.1}. 
Next, suppose we have the descent direction ${\p}^{(k)}$. 
The conjugate gradient vector of $E_A$ at $f^{(k+1)}$ is formulated as
\begin{equation} \label{eq:CG_directionK}
{\p}^{(k+1)} = -{\g}^{(k+1)} + \beta_{k+1} {\p}^{(k)}, ~ \textrm{ with $\beta_{k+1} = \frac{{{\g}^{(k+1)}}^\top {\g}^{(k+1)}}{{\g^{(k)}}^\top \g^{(k)}}$.}
\end{equation}
\end{subequations}
The map is then updated along with the conjugate gradient direction as
$$
\widehat{\f}^{(k+1)} = \widehat{\f}^{(k)} + \alpha_k {\p}^{(k)},
$$
where
$$
{\p}^{(k+1)} = -{\g}^{(k+1)} + \beta_{k+1} {\p}^{(k)}
$$
with ${\g}^{(k+1)} = \nabla \widetilde{E}_A(\widehat{\f}^{(k+1)})$ as in \eqref{eq:GradEA}.

\subsection{Quadratic approximation of step sizes}
\label{subsec:3.1}

To estimate the optimal step size $\alpha_k$, we write $\Phi^{(k)}(\alpha) = \widetilde{E}_A(\widehat{\f}^{(k)} + \alpha {\p}^{(k)})$. Then, by the chain rule, 
\begin{align*}
\frac{\d}{\d\alpha}\Phi^{(k)}(\alpha) 
&= {{\p}^{(k)}}^\top \nabla \widetilde{E}_A\big(\widehat{\f}^{(k)} + \alpha {\p}^{(k)}\big). 
\end{align*}
The critical point of the function $\Phi^{(k)}(\alpha)$ satisfies $\frac{\d}{\d\alpha}\Phi^{(k)}(\alpha) = 0$. However, deriving an explicit formulation for $\frac{\d}{\d\alpha}\Phi^{(k)}(\alpha)$ poses a challenge due to the dependence of the matrix $L_S(\widehat{\f}^{(k)} + \alpha {\p}^{(k)})$ on $\alpha$ as well. 
To remedy this drawback, we approximate the function $\Phi^{(k)}$ by a quadratic function 
\begin{equation} \label{eq:Phit}
\widetilde{\Phi}^{(k)}(\alpha) = a_2^{(k)} \alpha^2 + a_1^{(k)} \alpha + a_0^{(k)}
\end{equation}
that satisfies 
\begin{equation} \label{eq:Phi_cond}
\begin{cases}
\widetilde{\Phi}^{(k)}(0) = \Phi^{(k)}(0) = \widetilde{E}_A\big(\widehat{\f}^{(k)}\big), \\
\frac{\d}{\d\alpha}\widetilde{\Phi}^{(k)}(0) = \frac{\d}{\d\alpha}\Phi^{(k)}(0) 
= {{\p}^{(k)}}^\top \nabla \widetilde{E}_A\big(\widehat{\f}^{(k)}\big), \\
\widetilde{\Phi}^{(k)}(\alpha_{k-1}) = \Phi^{(k)}(\alpha_{k-1}) = \widetilde{E}_A\big(\widehat{\f}^{(k)} + \alpha_{k-1} {\p}^{(k)}\big).
\end{cases}
\end{equation}
By \eqref{eq:Phi_cond}, the coefficients $a_0^{(k)}$, $a_1^{(k)}$ and $a_2^{(k)}$ in \eqref{eq:Phit} can be explicitly formulated as 
\begin{align*}
a_0^{(k)} &= \Phi^{(k)}(0), ~~~
a_1^{(k)} = \frac{\d}{\d\alpha}\Phi^{(k)}(0), ~\text{ and }~\\
a_2^{(k)} &= \frac{1}{\alpha_{k-1}^2} \Big( \Phi^{(k)}(\alpha_{k-1}) - \Phi^{(k)}(0) - \alpha_{k-1}\frac{\d}{\d\alpha}\Phi^{(k)}(0) \Big).
\end{align*}
The step size $\alpha_k$ is then chosen to satisfy $\frac{\d}{\d\alpha}\widetilde{\Phi}(\alpha_k) = 0$, which implies
\begin{equation} \label{eq:alpha}
\alpha_k = \frac{-\frac{\d}{\d\alpha}\Phi^{(k)}(0)}{2 a_2^{(k)}} = \frac{-\alpha_{k-1}^2\frac{\d}{\d\alpha}\Phi^{(k)}(0)}{2\big( \Phi^{(k)}(\alpha_{k-1}) - \Phi^{(k)}(0) - \alpha_{k-1}\frac{\d}{\d\alpha}\Phi^{(k)}(0) \big)}.
\end{equation}
Based on observations from numerical experiments, we choose $\alpha_0=2$ as the initial step length in practice.

\subsection{Preconditioner}
\label{sec:3.2}

To accelerate the convergence of the nonlinear CG method, we incorporate a preconditioner
$$
M = \begin{bmatrix}
I_2 \otimes [L_S(f^{(0)})]_{\I,\I} & \\
& [L_S(f^{(0)})]_{\B,\B}
\end{bmatrix}.
$$
Then, the conjugate gradient direction \eqref{eq:CG_direction} is replaced with
\begin{subequations} \label{eq:PCG}
\begin{equation} \label{eq:PCG_0}
\p^{(0)} = -M^{-1} \g^{(0)}
\end{equation}
and
\begin{equation} \label{eq:PCG_k}
\p^{(k+1)} = -M^{-1} \g^{(k+1)} + \beta_{k+1} \p^{(k)}, ~ \textrm{ with $\beta_{k+1} = \frac{{{\g}^{(k+1)}}^\top M^{-1} {\g}^{(k+1)}}{{\g^{(k)}}^\top M^{-1} \g^{(k)}}$.}
\end{equation}
\end{subequations}
In practice, for each block $M_i$ of $M$, $i=1,2$, we compute the preordered Cholesky decomposition
\begin{equation} \label{eq:Chol}
U_i^\top U_i = P_i^\top M_i P_i,
\end{equation}
where $U_i$ is an upper triangular matrix and $P_i$ is the approximate minimum degree permutation matrix \cite{AmDD04}. Then, linear systems of the form $M_i\x = \r$ can be efficiently solved by solving lower and upper triangular systems 
\begin{subequations} \label{eq:LS}
\begin{equation} \label{eq:LS1}
U_i^\top \y = P_i^\top \r \,\text{ and }\, U_i\z = \y.
\end{equation}
Ultimately, the solution to the system is given by 
\begin{equation} \label{eq:LS2}
\x = P_i\z. 
\end{equation}
\end{subequations}

The detailed computational procedure is summarized in Algorithm \ref{alg:PCG}.

\begin{algorithm}[htbp]
\caption{Preconditioned Nonlinear CG method of AEM}
\label{alg:PCG}
\begin{algorithmic}[1]
\Require A simply connected open mesh $\mathcal{S}$.
\Ensure An area-preserving map $f:\mathcal{S}\to\mathbb{D}$.
\State Let $\I$ and $\B$ be defined in \eqref{eq:BI}.
\State Compute the arc-length parameterization $\b$ of $\partial\S$.
\State Define $\f_\B=\b$.
\For {$k=1, \ldots, 5$}
    \State Solve $L_{\I,\I} \f_\I = -L_{\I,\B} \b$.
    \State Update $L = L_S(f)$ as in \eqref{eq:L_S}.
\EndFor
\State Let $M_1 = L_{\I,\I}$ and $M_2 = L_{\B,\B}$.
\State Compute the Cholesky decomposition $U_i^\top U_i = P_i^\top M_i P_i$, as in \eqref{eq:Chol}, $i=1,2$.
\State Let $\g = \nabla \widetilde{E}_A(\f)$ as in \eqref{eq:GradEA}.
\State Solve $M\h = \g$ by \eqref{eq:LS}.
\State Let $\p = -\h$.
\State Define an initial guess $\alpha$.
\While {not converge}
    \State Update $\alpha$ as in \eqref{eq:alpha}.
    \State Update $\f \leftarrow \f + \alpha \p$.
    \State Update $L \gets L_S(f)$.
    \State Let $\lambda = \h^\top \g$.
    \State Update $\g = \nabla \widetilde{E}_A(\f)$ as in \eqref{eq:GradEA}.
    \State Solve $M\h = \g$ by \eqref{eq:LS}.
    \State Update $\beta = (\h^{\top} \g) / \lambda$.
    \State Update $\p \gets -\h + \beta \p$.
\EndWhile
\end{algorithmic}
\end{algorithm}

\section{Convergence analysis}
\label{sec:4} 

In this section, we prove the global convergence of the preconditioned CG method of the AEM Algorithm \ref{alg:PCG}.

First, we introduce the strong Wolfe conditions
\begin{subequations} \label{eq:Wolfe}
\begin{align}
\widetilde{E}_A(\widehat{\f}^{(k)} + \alpha_k \p^{(k)}) & \leq \widetilde{E}_A(\widehat{\f}^{(k)}) + c_1 \alpha_k \nabla \widetilde{E}_A(\widehat{\f}^{(k)})^{\top} \p_k, \label{eq:Wolfe1} \\
|\nabla \widetilde{E}_A(\widehat{\f}^{(k)} + \alpha_k \p^{(k)})^{\top} \p^{(k)}| & \leq c_2 | \nabla \widetilde{E}_A(\widehat{\f}^{(k)})^{\top} \p^{(k)}|, \label{eq:Wolfe2}
\end{align}
\end{subequations}
with $0 < c_1 < c_2 < 1$, which play a vital role in proving the convergence of Algorithm \ref{alg:PCG}. 
Since $\widetilde{E}_A(f) \geq 0$ is bounded below, the existence of such $\alpha_k$ is guaranteed by the following lemma.
\begin{lemma}[{\cite[Lemma 3.1]{NoWr06}}]
\label{lma:Wolfe}
Let $\p^{(k)}$ be a descent direction at $\widehat{\f}^{(k)}$. Suppose $0 < c_1 < c_2 < 1$. Then, there exist intervals of step lengths satisfying the strong Wolfe conditions \eqref{eq:Wolfe}.
\end{lemma}

Observing that Lemma \ref{lma:Wolfe} requires the assumption that $\p^{(k)}$ is a descent direction at $\widehat{\f}^{(k)}$. When a preconditioner $M$ is considered, a vital inequality is required to ensure that $\p^{(k)}$ is always a descent direction at $\widehat{\f}^{(k)}$. The inequality is stated in the following lemma, a variant of \cite[Lemma 5.6]{NoWr06}.

\begin{lemma}
\label{lma:descent}
Given a preconditioner $M$ with a symmetric positive definite inverse. Suppose the step size $\alpha_k$ satisfies the strong Wolfe conditions \eqref{eq:Wolfe}. Then, the vector $\p^{(k)}$ satisfies 
\begin{equation} \label{eq:descent_cond}
-\frac{1}{1-c_2} \leq \frac{\nabla \widetilde{E}_A(\widehat{\f}^{(k)})^{\top} \p^{(k)}}{\| \nabla \widetilde{E}_A(\widehat{\f}^{(k)}) \|^2_{M^{-1}}} \leq \frac{2c_2 - 1}{1 - c_2},
\end{equation}
for all $k \geq 0$, where $\|\cdot\|_{M^{-1}}$ denotes the $M^{-1}$-norm defined as 
$\|\v\|_{M^{-1}} = \sqrt{\v^\top M^{-1}\v}$.
\end{lemma}
\begin{proof}
We prove \eqref{eq:descent_cond} by induction on $k$. 
For $k = 0$, from \eqref{eq:PCG_0}, we have
\begin{equation*}
\frac{\nabla \widetilde{E}_A(\widehat{\f}^{(0)})^{\top} \p^{(0)}}{\| \nabla \widetilde{E}_A(\widehat{\f}^{(0)}) \|^2_{M^{-1}}}
= \frac{-\nabla \widetilde{E}_A(\widehat{\f}^{(0)})^{\top} M^{-1} \g^{(0)}}{\| \nabla \widetilde{E}_A(\widehat{\f}^{(0)}) \|^2_{M^{-1}}}
= \frac{-\| \nabla \widetilde{E}_A(\widehat{\f}^{(0)}) \|^2_{M^{-1}}}{\| \nabla\widetilde{E}_A(\widehat{\f}^{(0)}) \|^2_{M^{-1}}}
= -1.
\end{equation*}
Next, we assume \eqref{eq:descent_cond} holds for some $k=n$. 
Then, from \eqref{eq:PCG_k}, we have
{\small\begin{align}
\frac{\nabla \widetilde{E}_A(\widehat{\f}^{(n+1)})^{\top} \p^{(n+1)}}{\| \nabla \widetilde{E}_A(\widehat{\f}^{(n+1)}) \|^2_{M^{-1}}} 
&= \frac{\nabla \widetilde{E}_A(\widehat{\f}^{(n+1)})^{\top} M^{-1} \g^{(n+1)}}{\| \nabla \widetilde{E}_A(\widehat{\f}^{(n+1)}) \|^2_{M^{-1}}} + \beta_{n+1} \frac{\nabla \widetilde{E}_A(\widehat{\f}^{(n+1)})^{\top} \p^{(n)}}{\| \nabla \widetilde{E}_A(\widehat{\f}^{(n+1)}) \|^2_{M^{-1}}} \nonumber \\
&= -1 + \frac{{{\g}^{(n+1)}}^\top M^{-1} {\g}^{(n+1)}}{{\g^{(n)}}^\top M^{-1} \g^{(n)}} \frac{\nabla \widetilde{E}_A(\widehat{\f}^{(n+1)})^{\top} \p^{(n)}}{\| \nabla \widetilde{E}_A(\widehat{\f}^{(n+1)}) \|^2_{M^{-1}}} \nonumber \\
&= -1 + \frac{\nabla \widetilde{E}_A(\widehat{\f}^{(n+1)})^{\top} \p^{(n)}}{\| \nabla \widetilde{E}_A(\widehat{\f}^{(n)}) \|^2_{M^{-1}}}. \label{eq:4.13} 
\end{align}}
By \eqref{eq:Wolfe2}, we have
$
|\nabla \widetilde{E}_A(\widehat{\f}^{(n+1)})^{\top} \p^{(n)}| \leq -c_2 \nabla \widetilde{E}_A(\widehat{\f}^{(n)})^{\top} \p^{(n)},
$ 
i.e., 
\begin{align}
c_2 \nabla \widetilde{E}_A(\widehat{\f}^{(n)})^{\top} \p^{(n)} \leq \nabla \widetilde{E}_A(\widehat{\f}^{(n+1)})^{\top} \p^{(n)} \leq -c_2 \nabla \widetilde{E}_A(\widehat{\f}^{(n)})^{\top} \p^{(n)}. \label{eq:4.15} 
\end{align}
As a result, from \eqref{eq:4.13} and \eqref{eq:4.15}, we have
\begin{align*}
-1 + c_2\frac{\nabla \widetilde{E}_A(\widehat{\f}^{(n)})^\top \p^{(n)}}{\| \nabla \widetilde{E}_A(\widehat{\f}^{(n)}) \|^2_{M^{-1}}}
&\leq \frac{\nabla \widetilde{E}_A(\widehat{\f}^{(n+1)})^\top \p^{(n+1)}}{\| \nabla \widetilde{E}_A(\widehat{\f}^{(n+1)}) \|^2_{M^{-1}}} \\
&\leq -1 - c_2\frac{\nabla \widetilde{E}_A(\widehat{\f}^{(n)})^{\top} \p^{(n)}}{\| \nabla \widetilde{E}_A(\widehat{\f}^{(n)}) \|^2_{M^{-1}}}.
\end{align*}
Then, from the induction hypothesis, we have
\begin{equation*}
-\frac{1}{1-c_2} = -1 - \frac{c_2}{1-c_2} \leq \frac{\nabla \widetilde{E}_A(\widehat{\f}^{(n+1)})^{\top} \p^{(n+1)}}{\| \nabla \widetilde{E}_A(\widehat{\f}^{(n+1)}) \|^2_{M^{-1}}} \leq -1 + \frac{c_2}{1-c_2} = \frac{2c_2 - 1}{1 - c_2}, 
\end{equation*}
which shows that \eqref{eq:descent_cond} holds for $k=n+1$.
\end{proof}

In particular, when $c_2$ in \eqref{eq:Wolfe2} satisfies $0<c_2<\frac{1}{2}$, Lemma \ref{lma:descent} guarantees the vector $\p^{(k)}$ being a descent direction at $\f^{(k)}$, which is stated in the following corollary.

\begin{corollary} \label{cor:p}
Suppose the step size $\alpha_k$ satisfies the strong Wolfe conditions \eqref{eq:Wolfe} with $0<c_2<\frac{1}{2}$. Then, $\p^{(k)}$ is a descent direction at $\widehat{\f}^{(k)}$.
\end{corollary}
\begin{proof}
Noting that the function $\varphi(x) = \frac{2x-1}{1-x}$ is monotonically increasing on $[0, \frac{1}{2}]$ with $\varphi(0) = -1$ and $\varphi(\frac{1}{2}) = 0$. 
From the assumption $0<c_2<\frac{1}{2}$ and Lemma \ref{lma:descent}, we have
$$
\frac{\nabla \widetilde{E}_A(\widehat{\f}^{(k)})^{\top} \p^{(k)}}{\| \nabla \widetilde{E}_A(\widehat{\f}^{(k)}) \|^2_{M^{-1}}} \leq \frac{2c_2 - 1}{1 - c_2} < 0.
$$
Hence $\nabla \widetilde{E}_A(\widehat{\f}^{(k)})^{\top} \p^{(k)}<0$, and therefore, $\p^{(k)}$ is a descent direction at $\widehat{\f}^{(k)}$.
\end{proof}

Noting that the normalized authalic energy $\widetilde{E}_A(\widehat{\f})\geq 0$ is bounded below, and that $\widetilde{E}_A$ is continuously differentiable. 
In addition, $\widetilde{E}_A$ is Lipschitz continuous. To extend the Lipschitz continuity in $M$-norm, we provide the following lemma.
\begin{lemma} \label{lma:Lip}
Let $M$ be a symmetric positive definite matrix and $f: \mathbb{R}^n \to \mathbb{R}^n$ be Lipschitz continuous with respect to the Euclidean norm, i.e.,
\[
\|f(\mathbf{x}) - f(\mathbf{y})\| \leq c_f \, \|\mathbf{x} - \mathbf{y}\|,
\]
for some constant $c_f \geq 0$. Then it also satisfies Lipschitz continuity in the $M$-norm, i.e.,
\[
\|f(\mathbf{x}) - f(\mathbf{y})\|_M \leq m_f \, \|\mathbf{x} - \mathbf{y}\|_M,
\]
for some constant $m_f\geq 0$.
\end{lemma}
\begin{proof}
Since $M$ is symmetric positive definite, $M$ is invertible and there exist symmetric positive definite matrices $M^{1/2}$ and $M^{-1/2}$ that satisfies $M = M^{1/2} M^{1/2}$ and $M^{-1} = M^{-1/2} M^{-1/2}$. It follows that
\[
\|\mathbf{x}\|_M^2 = \mathbf{x}^\top M \mathbf{x} = (M^{1/2} \mathbf{x})^\top M^{1/2} \mathbf{x} = \|M^{1/2} \mathbf{x}\|^2,
\]
and
\begin{equation} \label{eq:E2M}
\|\mathbf{x}\|^2 = \mathbf{x}^\top \mathbf{x} = (M^{-1/2} \mathbf{x})^\top M M^{-1/2} \mathbf{x} = \|M^{-1/2} \mathbf{x}\|_M^2.
\end{equation}
Therefore,
\begin{align*}
\|f(\mathbf{x}) - f(\mathbf{y})\|_M &= \| M^{1/2} (f(\mathbf{x}) - f(\mathbf{y})) \| 
\leq \| M^{1/2} \| \, \| f(\mathbf{x}) - f(\mathbf{y}) \| \\
& \leq c_f\, \| M^{1/2} \| \, \| \mathbf{x} - \mathbf{y} \| 
\leq c_f\, \| M^{1/2} \| \, \|M^{-1/2}\|_M \|\mathbf{x} - \mathbf{y}\|_M.
\end{align*}
The constant $m_f = c_f \,\| M^{1/2} \| \, \|M^{-1/2}\|_M \geq 0$.
\end{proof}

Under the assumption that $\alpha_k$ satisfies the Wolfe conditions \eqref{eq:Wolfe} with $0<c_2<\frac{1}{2}$ the modified Zoutendijk's condition would also hold, which is stated as follows.
\begin{lemma}[modified Zoutendijk's condition] \label{lma:Zou}
Suppose $\nabla \widetilde{E}_A$ is Lipschitz continuous, and the step size $\alpha_k$ satisfies the strong Wolfe conditions \eqref{eq:Wolfe} with $0<c_2<\frac{1}{2}$. 
Given a matrix $M$ with $M^{-1}$ being a symmetric positive definite matrix.
Then, the vector $\p^{(k)}$ satisfies
\begin{subequations} \label{eq:Zou}
\begin{equation}
\sum_{k=0}^{\infty} \cos^2 \theta_{M}^{(k)} \| \nabla \widetilde{E}_A(\widehat{\f}^{(k)}) \|_{M^{-1}}^2 < \infty,
\end{equation}
where $\theta^{(k)}_M$ is defined as
\begin{equation} \label{eq:cos}
\cos\theta_{M}^{(k)} = -\frac{\nabla \widetilde{E}_A(\widehat{\f}^{(k)})^\top \p^{(k)}}{\| \nabla \widetilde{E}_A(\widehat{\f}^{(k)}) \|_{M^{-1}} \| \p^{(k)} \|_{M^{-1}}}.
\end{equation}
\end{subequations}
\end{lemma}
\begin{proof}

Since the step size $\alpha_k$ satisfies \eqref{eq:Wolfe}, by Corollary \ref{cor:p}, $\p^{(k)}$ is a descent direction at $\widehat{\f}^{(k)}$ so that  \eqref{eq:Wolfe2} can be written as
\begin{equation} \label{eq:lma4.4_1}
\nabla \widetilde{E}_A(\widehat{\f}^{(k+1)})^{\top} \p^{(k)} \geq c_2 \nabla \widetilde{E}_A(\widehat{\f}^{(k)})^{\top} \p^{(k)}.
\end{equation}
By subtracting $\nabla \widetilde{E}_A(\widehat{\f}^{(k)})^{\top} \p^{(k)}$ in both sides of \eqref{eq:lma4.4_1}, we obtain
\begin{equation*}
\big(\nabla \widetilde{E}_A(\widehat{\f}^{(k+1)}) - \nabla \widetilde{E}_A(\widehat{\f}^{(k)})\big)^{\top} \p^{(k)} \geq (c_2-1) \nabla \widetilde{E}_A(\widehat{\f}^{(k)})^{\top} \p^{(k)}.
\end{equation*}
By applying Lemma \ref{lma:Lip} and \eqref{eq:E2M}, we have
\begin{align*}
(c_2-1) \nabla \widetilde{E}_A(\widehat{\f}^{(k)})^{\top} \p^{(k)}
& \leq \big( \nabla \widetilde{E}_A(\widehat{\f}^{(k+1)}) - \nabla \widetilde{E}_A(\widehat{\f}^{(k)}) \big)^{\top} \p^{(k)} \\
& \leq \| \nabla \widetilde{E}_A(\widehat{\f}^{(k+1)}) - \nabla \widetilde{E}_A(\widehat{\f}^{(k)}) \| \|\p^{(k)}\| \\
&\leq c_E \, \| \widehat{\f}^{(k+1)} - \widehat{\f}^{(k)} \| \| \p^{(k)}\| \\
&= c_E \, \| M^{1/2} (\widehat{\f}^{(k+1)} - \widehat{\f}^{(k)}) \|_{M^{-1}} \| M^{1/2} \p^{(k)}\|_{M^{-1}} \\
& \leq c_E \, \|M^{1/2}\|_{M^{-1}}^2 \| \widehat{\f}^{(k+1)} - \widehat{\f}^{(k)} \|_{M^{-1}}  \| \p^{(k)}\|_{M^{-1}} \\
&= c_E \, \|M^{1/2}\|_{M^{-1}}^2 \alpha_k  \| \p^{(k)}\|_{M^{-1}}^2.
\end{align*}
As a result,
\begin{align*}
\alpha_k &\geq \frac{c_2-1}{c_E \, \|M^{1/2}\|_{M^{-1}}^2} \frac{\nabla \widetilde{E}_A(\widehat{\f}^{(k)})^{\top} \p^{(k)}}{\| \p^{(k)}\|^2_{M^{-1}}}.
\end{align*}
By substituting $\alpha_k$ into \eqref{eq:Wolfe1}, we obtain
\begin{align}
\widetilde{E}_A(\widehat{\f}^{(k+1)}) &\leq \widetilde{E}_A(\widehat{\f}^{(k)}) - \frac{c (\nabla \widetilde{E}_A(\widehat{\f}^{(k)})^{\top} \p^{(k)})^2}{\| \p^{(k)}\|^2_{M^{-1}}} \\
&= \widetilde{E}_A(\widehat{\f}^{(k)}) - c \cos^2\theta_M^{(k)} \|\widetilde{E}_A(\widehat{\f}^{(k)})\|_{M^{-1}}^2, \label{eq:4.12}
\end{align}
where $c=\frac{c_1(1-c_2)}{c_E \, \|M^{1/2}\|_{M^{-1}}^2}$ and $\cos\theta_M^{(k)}$ is defined as \eqref{eq:cos}. 
The recursive inequality \eqref{eq:4.12} implies
$$
\widetilde{E}_A(\widehat{\f}^{(k+1)}) \leq \widetilde{E}_A(\widehat{\f}^{(0)}) - c \sum_{j=0}^k \cos^2\theta_M^{(j)} \|\widetilde{E}_A(\widehat{\f}^{(j)})\|_{M^{-1}}^2.
$$
In other words,
\begin{equation} \label{eq:3.15}
\sum_{j = 0}^{k} \cos^2 \theta_{M}^{(j)} \| \nabla \widetilde{E}_A(\widehat{\f}^{(j)}) \|_{M^{-1}}^2 = \frac{1}{c}\big( \widetilde{E}_A(\widehat{\f}^{(0)}) - \widetilde{E}_A(\widehat{\f}^{(k+1)}) \big).
\end{equation}
Noting that $\widetilde{E}_A(\widehat{\f})$ is bounded below, which implies $\widetilde{E}_A(\widehat{\f}^{(0)}) - \widetilde{E}_A(\widehat{\f}^{(k+1)})$ is bounded above. Therefore, by taking the limit of \eqref{eq:3.15}, we obtain \eqref{eq:Zou} as desired.
\end{proof}

Lemma \ref{lma:Zou} implies that as $\cos^2 \theta_M^{(k)} \| \nabla \widetilde{E}_A(\widehat{\f}^{(k)}) \|_{M^{-1}}^2$ tends to zero, the gradient descent method converges because $\theta_M^{(k)} = 0$ and $\cos\theta_M^{(k)} = 1$. In the case of the nonlinear preconditioned CG method, where $\p^{(k)}$ is a descent direction with $\theta_M^{(k)} < \frac{\pi}{2}$, we have $\cos\theta_M^{(k)} > 0$ for all $k$. However, it is still possible that $\cos\theta_M^{(k)}$ approaches zero, which means that $\| \nabla \widetilde{E}_A(\widehat{\f}^{(k)}) \|_{M^{-1}}$ may not converge to zero. Fortunately, we can apply Lemmas \ref{lma:descent} and \ref{lma:Zou} to establish a global convergence theorem for the nonlinear preconditioned CG method, which is stated as follows.

\begin{theorem}
Suppose $\nabla \widetilde{E}_A$ is Lipschitz continuous, and the step size $\alpha_k$ satisfies the strong Wolfe conditions \eqref{eq:Wolfe} with $0<c_1<c_2<\frac{1}{2}$. Then,
\begin{equation} \label{eq:PCG_Convergence}
\liminf_{k \to \infty}{\| \nabla \widetilde{E}_A(\widehat{\f}^{(k)}) \|_{M^{-1}} } = 0,
\end{equation}
where $M^{-1}$ is a symmetric positive definite matrix.
\end{theorem}
\begin{proof}
We prove \eqref{eq:PCG_Convergence} by contradiction.
Assume on the contrary that 
$$
\liminf_{k\to\infty}{\| \nabla \widetilde{E}_A(\widehat{\f}^{(k)}) \|_{M^{-1}}} \neq 0,
$$
i.e., there exist $N\in\mathbb{N}$ and $\gamma > 0$ such that 
\begin{equation}\label{eq:1}
\| \nabla \widetilde{E}_A(\widehat{\f}^{(k)}) \|_{M^{-1}} \geq \gamma,
\end{equation}
for every $k>N$. Let $\theta_M^{(k)}$ be defined as \eqref{eq:cos}. 
Then, by Lemma \ref{lma:descent} together with \eqref{eq:cos}, we have
\begin{equation*}
-\frac{1}{1-c_2} \leq -\cos\theta_M^{(k)} \frac{\| \p^{(k)} \|_{M^{-1}}}{\| \nabla \widetilde{E}_A(\widehat{\f}^{(k)}) \|_{M^{-1}}} \leq \frac{2c_2 - 1}{1 - c_2}.
\end{equation*}
Equivalently, 
\begin{equation}
\frac{1-2c_2}{1-c_2} \frac{\| \nabla \widetilde{E}_A(\widehat{\f}^{(k)}) \|_{M^{-1}}}{\| \p^{(k)} \|_{M^{-1}}} \leq \cos\theta_k  \leq \frac{1}{1 - c_2} \frac{\| \nabla \widetilde{E}_A(\widehat{\f}^{(k)}) \|_{M^{-1}}}{\| \p^{(k)} \|_{M^{-1}}}. \label{eq:2}
\end{equation}
By Lemma \ref{lma:Zou}, we have
\begin{equation}
\sum_{k=0}^{\infty} \cos^2\theta_M^{(k)} \| \nabla \widetilde{E}_A(\widehat{\f}^{(k)}) \|^2_{M^{-1}} < \infty.  \label{eq:3}
\end{equation}
Hence, with \eqref{eq:2} and \eqref{eq:3}, we obtain
\begin{equation}
\sum_{k=0}^{\infty} \frac{\| \nabla \widetilde{E}_A(\widehat{\f}^{(k)}) \|^4_{M^{-1}}}{\| \p^{(k)} \|^2_{M^{-1}}} < \infty.  \label{eq:4}
\end{equation}
By Lemma \ref{lma:descent}, we have
\begin{equation} \label{eq:4.17a}
-\nabla \widetilde{E}_A(\widehat{\f}^{(k)})^{\top} \p^{(k)} \leq \frac{1}{1-c_2} \| \nabla \widetilde{E}_A(\widehat{\f}^{(k)}) \|^2_{M^{-1}}.
\end{equation}
Then, by \eqref{eq:Wolfe2} and \eqref{eq:4.17a}, we obtain
\begin{equation}
| \nabla \widetilde{E}_A(\widehat{\f}^{(k)})^{\top} \p^{(k-1)} | \leq -c_2 \nabla \widetilde{E}_A(\widehat{\f}^{(k-1)})^{\top} \p^{(k-1)} \leq \frac{c_2}{1-c_2} \| \nabla \widetilde{E}_A(\widehat{\f}^{(k-1)}) \|^2_{M^{-1}}.  \label{eq:5}
\end{equation}
Hence, by \eqref{eq:5}, we have
{\small\begin{align}
&\| \p^{(k)} \|^2_{M} \nonumber \\
&= \| -M^{-1} \nabla \widetilde{E}_A(\widehat{\f}^{(k)}) + \beta_k \p^{(k-1)} \|^2_{M} \nonumber\\
&=  \| M^{-1} \nabla \widetilde{E}_A(\widehat{\f}^{(k)}) \|_{M}^2 - 2 \beta_k (M^{-1} \nabla \widetilde{E}_A(\widehat{\f}^{(k)}) )^\top (M \p^{(k-1)}) + \beta^2_k \|\p^{(k-1)}\|_{M}^2 \nonumber\\
& \leq \| \nabla \widetilde{E}_A(\widehat{\f}^{(k)}) \|^2_{M^{-1}} + 2\beta_k | \nabla \widetilde{E}_A(\widehat{\f}^{(k)})^{\top} \p^{(k-1)} | + \beta^2_k \|\p^{(k-1)}\|_{M}^2 \nonumber\\
&\leq \| \nabla \widetilde{E}_A(\widehat{\f}^{(k)}) \|^2_{M^{-1}} + \frac{2c_2}{1-c_2} \beta_k \| \nabla \widetilde{E}_A(\widehat{\f}^{(k-1)}) \|^2_{M^{-1}} + \beta^2_k \|\p^{(k-1)}\|_{M}^2 \nonumber\\
&= \| \nabla \widetilde{E}_A(\widehat{\f}^{(k)}) \|^2_{M^{-1}} + \frac{2c_2}{1-c_2} \| \nabla \widetilde{E}_A(\widehat{\f}^{(k)}) \|^2_{M^{-1}}  + \beta^2_k \|\p^{(k-1)}\|_{M}^2 \nonumber\\
&= \Big (\frac{1+c_2}{1-c_2} \Big ) \| \nabla \widetilde{E}_A(\widehat{\f}^{(k)}) \|^2_{M^{-1}} + \beta^2_k \|\p^{(k-1)}\|_{M}^2. \label{eq:6}
\end{align}}
Let $c_3 = (1+c_2) / (1-c_2)$.
The recursive inequality \eqref{eq:6} with $\p^{(0)} = M^{-1} \nabla \widetilde{E}_A(\widehat{\f}^{(0)})$ implies
\begin{equation} \label{eq:p_k}
\| \p^{(k)} \|^2_{M} \leq c_3 \sum_{i=0}^k \Big(\prod_{j=i+1}^k \beta_j^2\Big) \|\nabla \widetilde{E}_A(\widehat{\f}^{(i)})\|_{M^{-1}}^2.
\end{equation}
Recall that $\beta_k = \| \nabla \widetilde{E}_A(\widehat{\f}^{(k)}) \|^2_{M^{-1}} / \| \nabla \widetilde{E}_A(\widehat{\f}^{(k-1)}) \|^2_{M^{-1}}$, the product
\begin{equation} \label{eq:7}
\prod_{j=i+1}^k \beta_j^2 = \frac{\| \nabla \widetilde{E}_A(\widehat{\f}^{(k)}) \|^4_{M^{-1}}}{\| \nabla \widetilde{E}_A(\widehat{\f}^{(i)}) \|^4_{M^{-1}}}.  
\end{equation}
By substituting \eqref{eq:7} into \eqref{eq:p_k}, we obtain
\begin{equation} \label{eq:4.31}
\| \p^{(k)} \|^2_{M} \leq c_3 \| \nabla \widetilde{E}_A(\widehat{\f}^{(k)}) \|^4_{M^{-1}} \sum_{i=0}^k \frac{1}{\| \nabla \widetilde{E}_A(\widehat{\f}^{(i)}) \|^2_{M^{-1}}}.
\end{equation}
In addition, by \eqref{eq:1} and the inequality $\| \nabla \widetilde{E}_A(\widehat{\f}^{(k)}) \|^2_{M^{-1}} \leq \| \nabla \widetilde{E}_A(\widehat{\f}^{(j)}) \|^2_{M^{-1}}$ for $j \leq k$, we have, for $k>N$,
\begin{equation}
\sum_{i=0}^k \frac{1}{\| \nabla \widetilde{E}_A(\widehat{\f}^{(i)}) \|^2_{M^{-1}}} \leq \sum_{i=0}^k \frac{1}{\| \nabla \widetilde{E}_A(\widehat{\f}^{(k)}) \|^2_{M^{-1}}}  \leq \frac{k}{\gamma^2}.
\label{eq:4.32}
\end{equation}
Since $\nabla \widetilde{E}_A$ is Lipschitz continuous, there exists $\overline{\gamma}$ such that $\| \nabla \widetilde{E}_A(\widehat{\f}^{(k)}) \|_{M^{-1}} \leq \overline{\gamma}$ for all $k$.
As a result, from \eqref{eq:4.31} and \eqref{eq:4.32}, for $k>N$,
\begin{equation*}
\| \p^{(k)} \|^2_{M} \leq \frac{c_3 \overline{\gamma}^4}{\gamma^2} k.
\end{equation*}
This implies that
\begin{equation} \label{eq:8}
\sum_{k=1}^{\infty} \frac{1}{\| \p^{(k)} \|^2_{M}} \geq \omega \sum_{k=1}^{\infty} \frac{1}{k}, 
\end{equation}
for some $\omega>0$.
On the other hand, from \eqref{eq:1} and \eqref{eq:4}, we have
\begin{align*}
\sum_{k=0}^{\infty} \frac{1}{\| \p^{(k)} \|^2_{M}} < \infty, 
\end{align*}
which contradicts to \eqref{eq:8}.
\end{proof}

\section{Numerical experiments}
\label{sec:5}

We present the numerical results of the disk area-preserving parameterizations computed by the preconditioned CG method for AEM, Algorithm \ref{alg:PCG}. 
The experiments were performed using MATLAB on a MacBook Pro M1 Max with 32 GB RAM. The benchmark triangular mesh models, shown in Figure \ref{fig:MeshModel}, were acquired from reputable sources such as the AIM@SHAPE shape repository \cite{AIM}, the Stanford 3D scanning repository \cite{Stanford}, and Sketchfab \cite{Sketchfab}. Some mesh models were modified to guarantee that each triangular face contains at least one interior vertex.

\begin{figure}
\centering
\resizebox{\textwidth}{!}{
\begin{tabular}{cccccccc}
\cline{1-2}\cline{4-5}\cline{7-8}
\multicolumn{2}{c}{Foot} && \multicolumn{2}{c}{Chinese Lion} && \multicolumn{2}{c}{Femur} \\
$\F(\M) = 19,966$ & $\V(\M) = 10,010$ && $\F(\M) = 34,421$ & $\V(\M) = 17,334$ && $\F(\M) = 43,301$ & $\V(\M) = 21,699$ \\
\includegraphics[height=3cm]{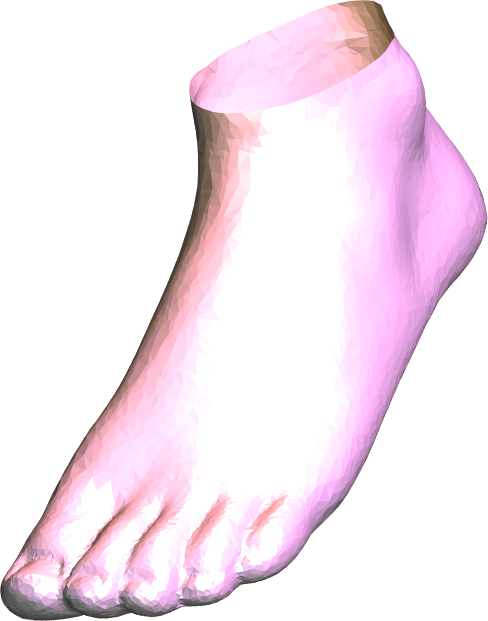} &
\includegraphics[height=3cm]{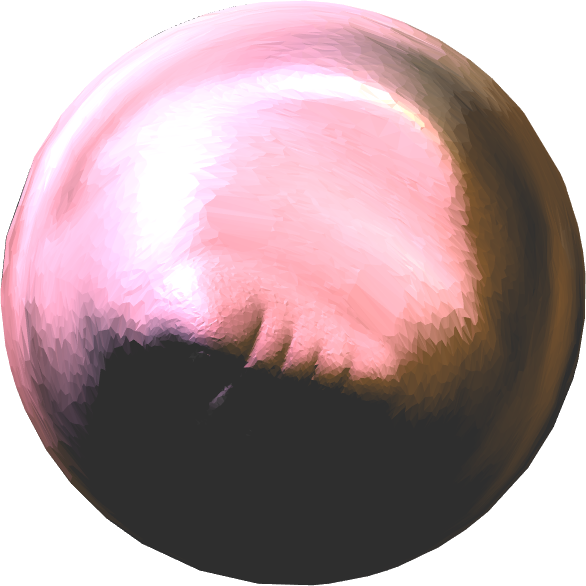} &&
\includegraphics[height=3cm]{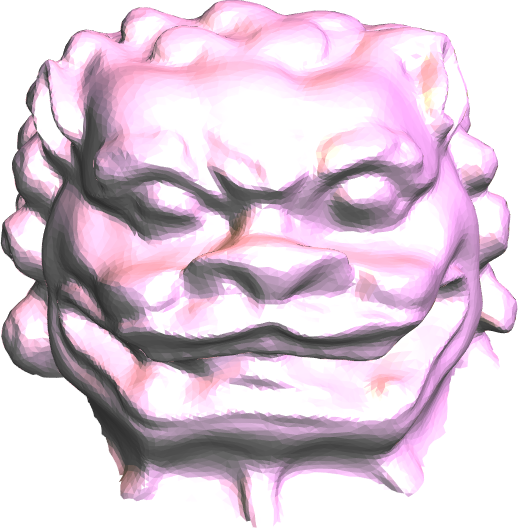} &
\includegraphics[height=3cm]{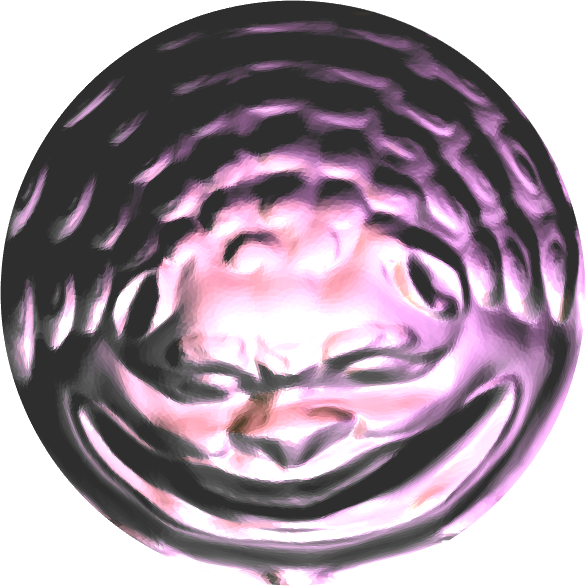} &&
\includegraphics[height=3cm]{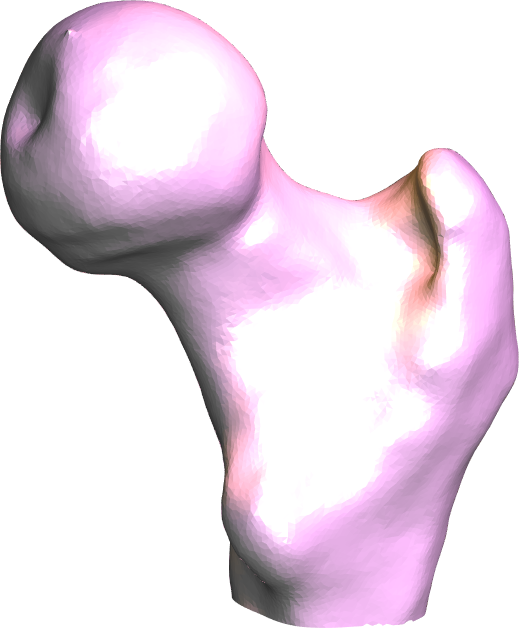} & 
\includegraphics[height=3cm]{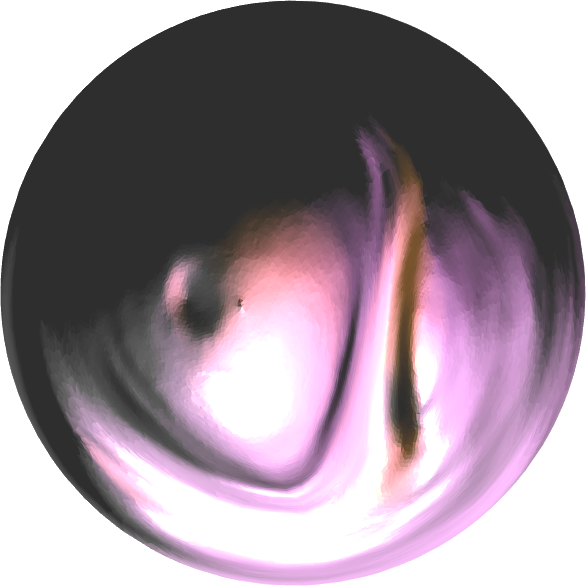} 
\\ \cline{1-2}\cline{4-5}\cline{7-8}
\\[-0.3cm]\cline{1-2}\cline{4-5}\cline{7-8}
\multicolumn{2}{c}{Stanford Bunny} && \multicolumn{2}{c}{Max Planck} && \multicolumn{2}{c}{Human Brain} \\
$\F(\M) = 62,946$ & $\V(\M) = 31,593$ && $\F(\M) = 82,977$ & $\V(\M) = 41,588$ && $\F(\M) = 96,811$ & $\V(\M) = 48,463$ \\
\includegraphics[height=3cm]{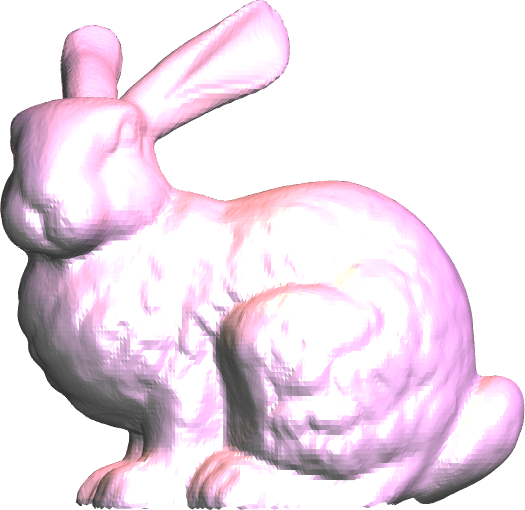} &
\includegraphics[height=3cm]{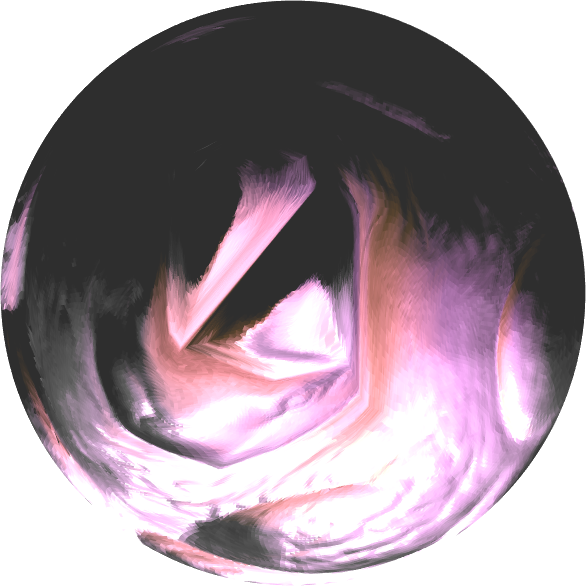} &&
\includegraphics[height=3cm]{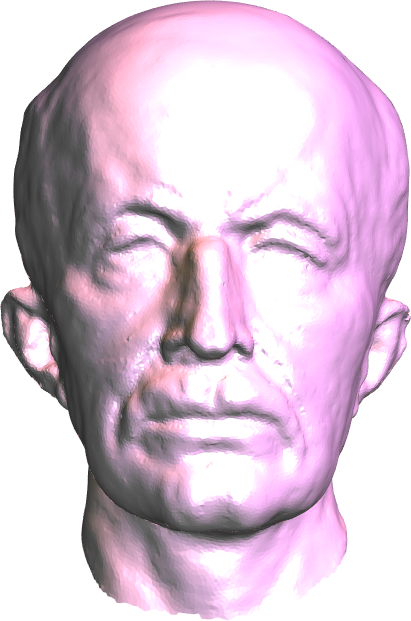} &
\includegraphics[height=3cm]{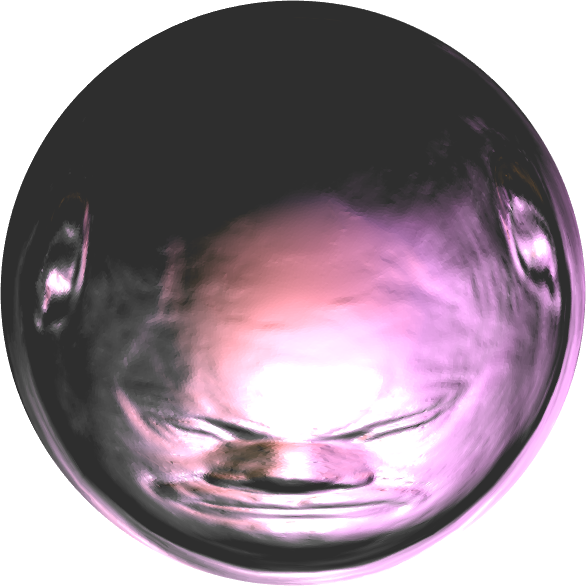} &&
\includegraphics[height=3cm]{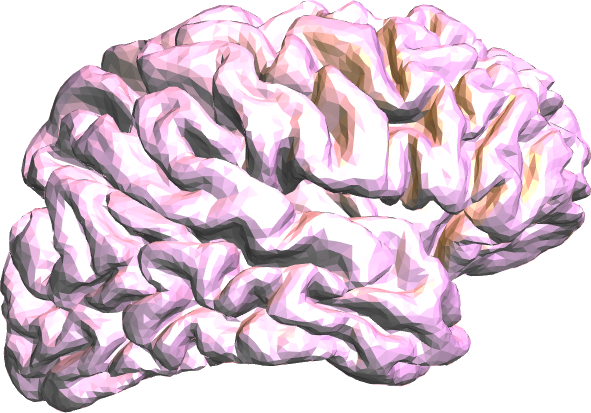} &
\includegraphics[height=3cm]{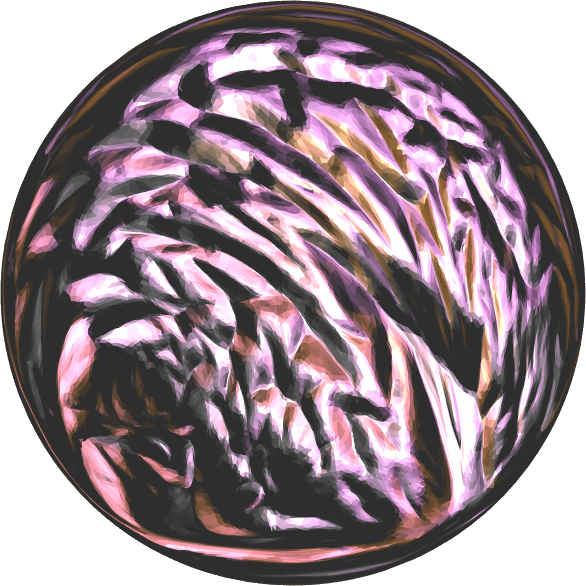} 
\\ \cline{1-2}\cline{4-5}\cline{7-8}
\\[-0.3cm]\cline{1-2}\cline{4-5}\cline{7-8}
\multicolumn{2}{c}{Left Hand} && \multicolumn{2}{c}{Knit Cap Man} && \multicolumn{2}{c}{Human Face} \\
$\F(\M) = 105,780$ & $\V(\M) = 53,011$ && $\F(\M) = 118,849$ & $\V(\M) = 59,561$ && $\F(\M) = 278,980$ & $\V(\M) = 140,085$ \\
\includegraphics[height=3cm]{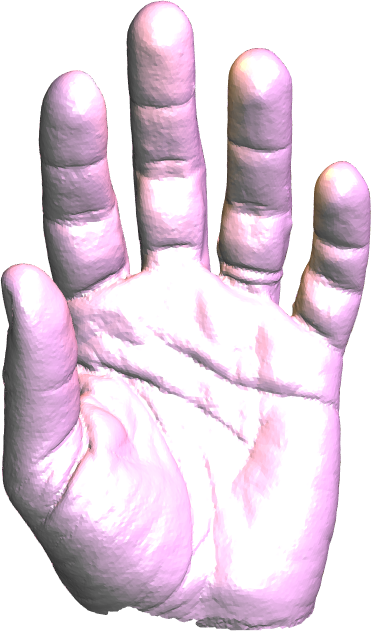} &
\includegraphics[height=3cm]{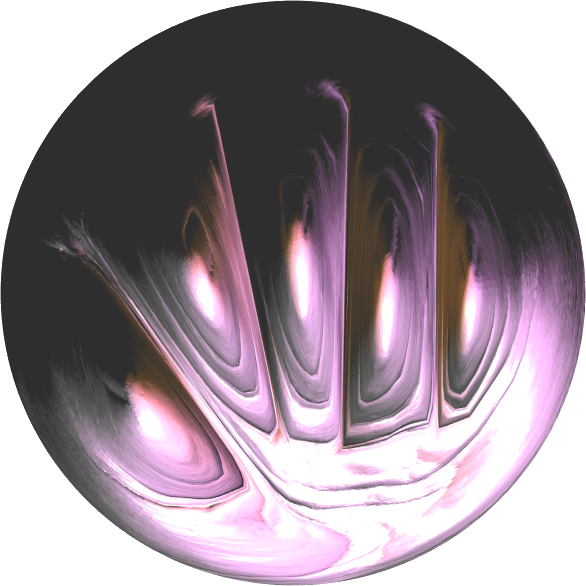} &&
\includegraphics[height=3cm]{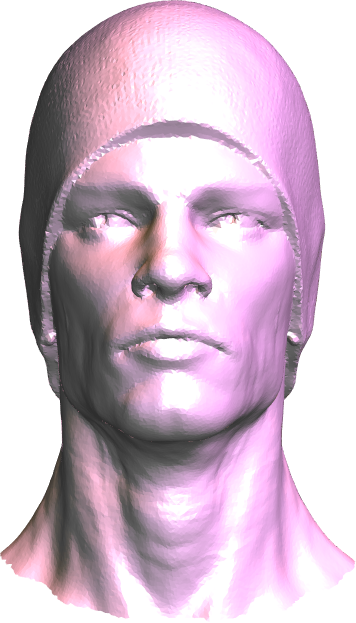} &
\includegraphics[height=3cm]{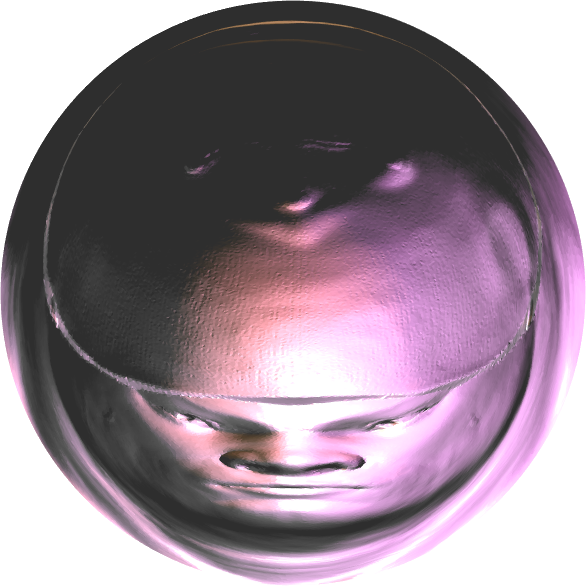} &&
\includegraphics[height=3cm]{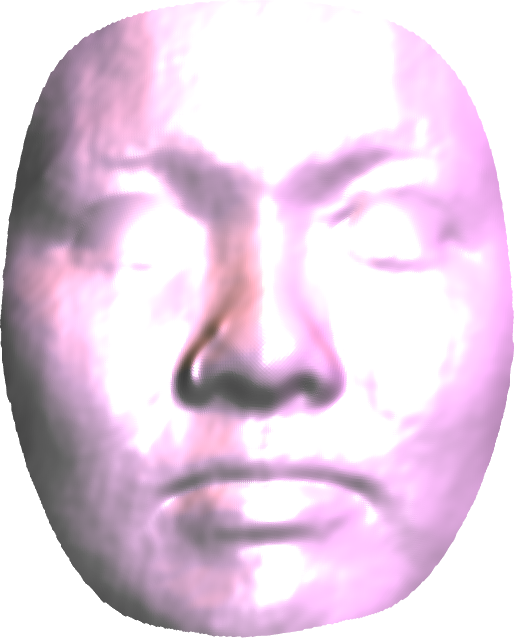} & 
\includegraphics[height=3cm]{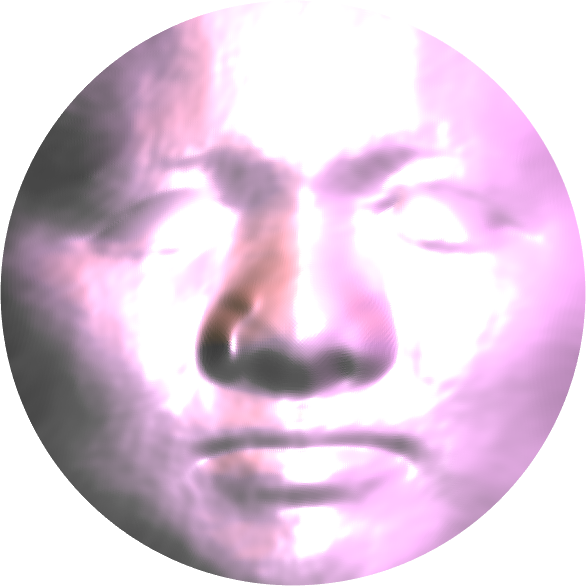} 
\\ \cline{1-2}\cline{4-5}\cline{7-8}
\\[-0.3cm]\cline{1-2}\cline{4-5}\cline{7-8}
\multicolumn{2}{c}{Bimba Statue} && \multicolumn{2}{c}{Buddha} && \multicolumn{2}{c}{Nefertiti Statue} \\
$\F(\M) = 433,040$ & $\V(\M) = 216,873$ && $\F(\M) = 945,722$ & $\V(\M) = 473,362$ && $\F(\M) = 1,992,801$ & $\V(\M) = 996,838$ \\
\includegraphics[height=3cm]{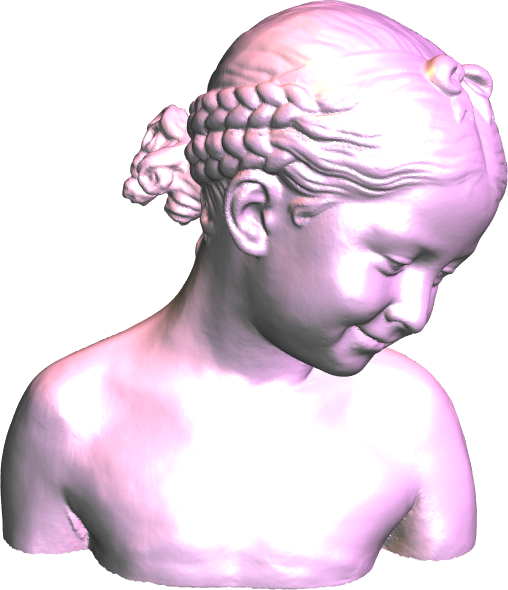} & 
\includegraphics[height=3cm]{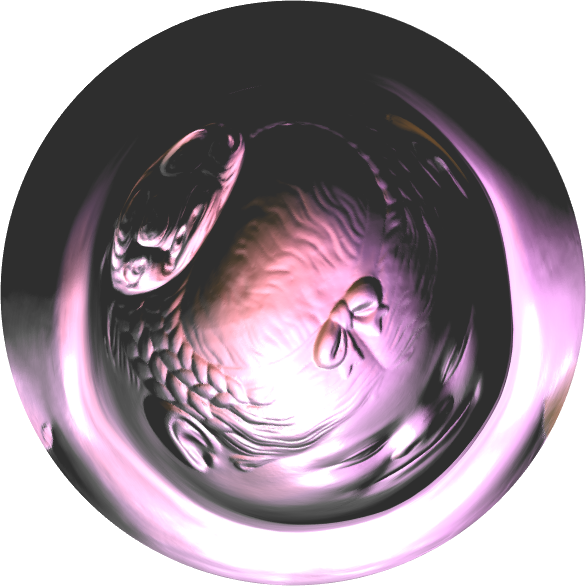} &&
\includegraphics[height=3cm]{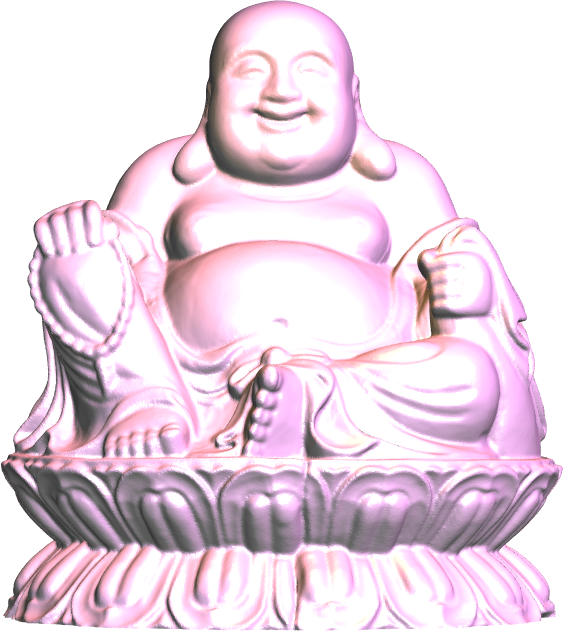} &
\includegraphics[height=3cm]{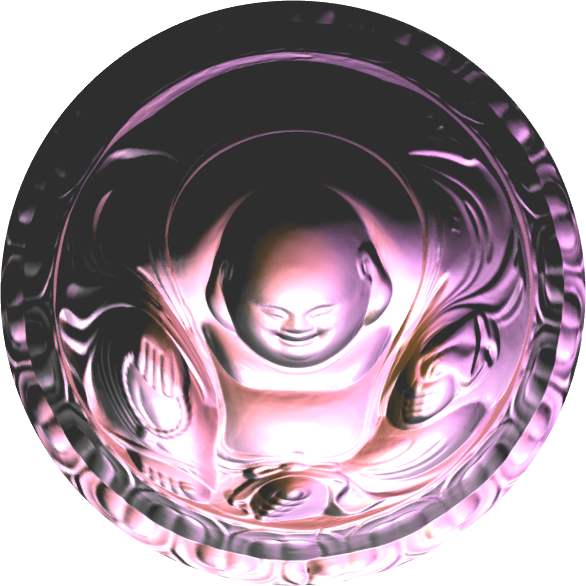} &&
\includegraphics[height=3cm]{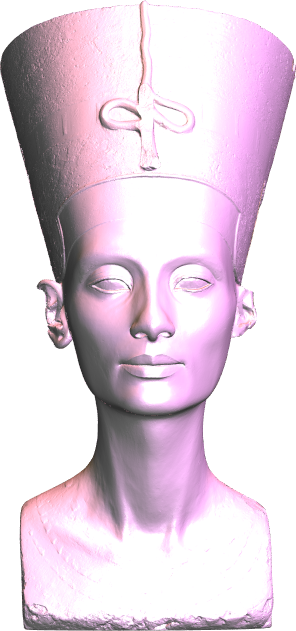} & 
\includegraphics[height=3cm]{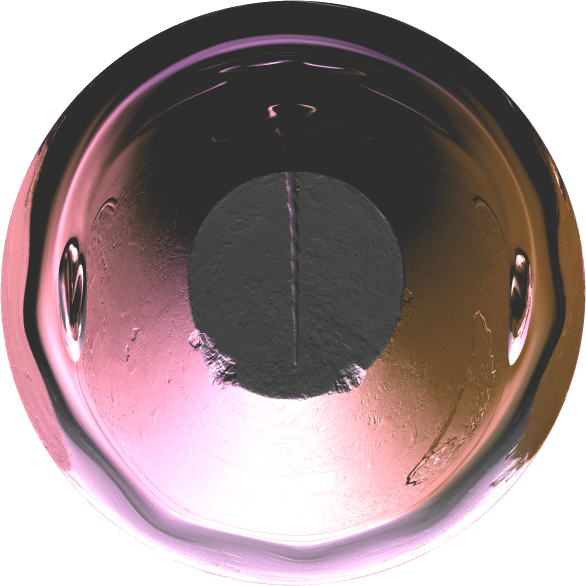} 
\\ \cline{1-2}\cline{4-5}\cline{7-8}
\end{tabular}
}
\caption{The benchmark triangular mesh models and their disk area-preserving parameterizations computed using the proposed nonlinear CG method of AEM, Algorithm \ref{alg:PCG}.}
\label{fig:MeshModel}
\end{figure}

In Figure \ref{fig:MeshModel}, we present images of area-preserving parameterizations computed by the proposed preconditioned nonlinear CG method of AEM, Algorithm \ref{alg:PCG}. It is important to emphasize that all parameterizations displayed in the figure are bijective, meaning no folded triangles occur in the resulting mappings. Consequently, each mapping can be utilized as a reliable surface parameterization for subsequent applications.

To quantify the local area distortions, we compute the standard deviation (SD) of the local area ratio, denoted as $R_A(f,\tau)$, defined as
\begin{equation} \label{eq:R_A}
R_A(f,\tau) = \frac{|f(\tau)|}{|\tau|},
\end{equation}
where $f$ is the computed area-preserving parameterization and $\tau$ represents a triangular face in the set $\F(\S)$. Evaluating the SD involves considering all triangular faces in $\F(\S)$. It is evident that an area-preserving mapping $f$ has the property that
$$
\underset{\tau\in\F(\S)}{\mathrm{mean}}R_A(f,\tau) = 1 
~\text{ and }~
\underset{\tau\in\F(\S)}{\mathrm{SD}}R_A(f,\tau) = 0. 
$$
Figure \ref{fig:hist} illustrates histograms of the distribution of area ratios $R_A(f,\tau)$, defined in \eqref{eq:R_A}, for the area-preserving parameterizations of benchmark mesh models computed by Algorithm \ref{alg:PCG}. The histograms reveal a concentrated distribution of area ratios around $1$, with radii measuring less than $0.2$. This observation indicates that the resulting mappings are nearly area-preserving.

\begin{figure}
\centering
\resizebox{\textwidth}{!}{
\begin{tabular}{cccc}
Foot & Chinese Lion & Femur & Stanford Bunny\\
\includegraphics[height=4.2cm]{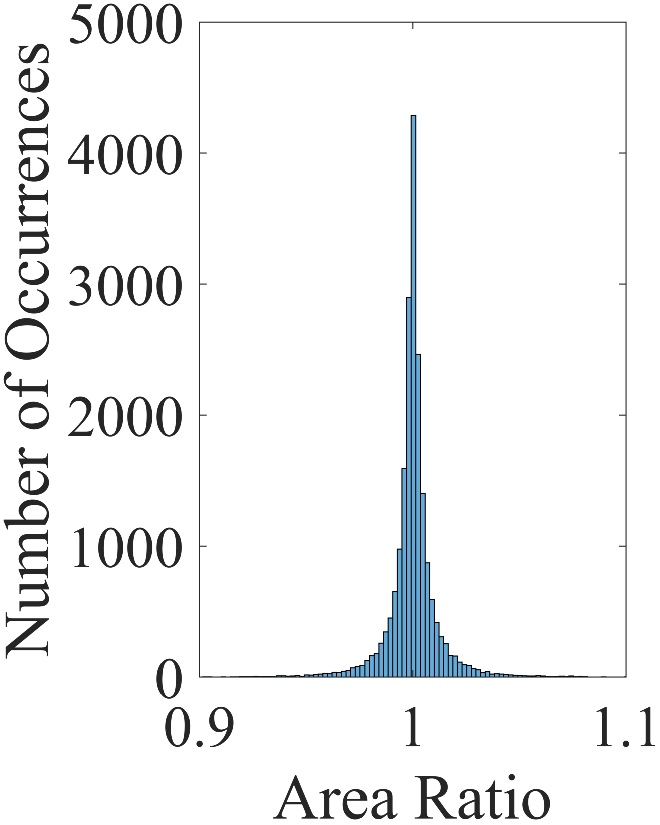} &
\includegraphics[height=4.2cm]{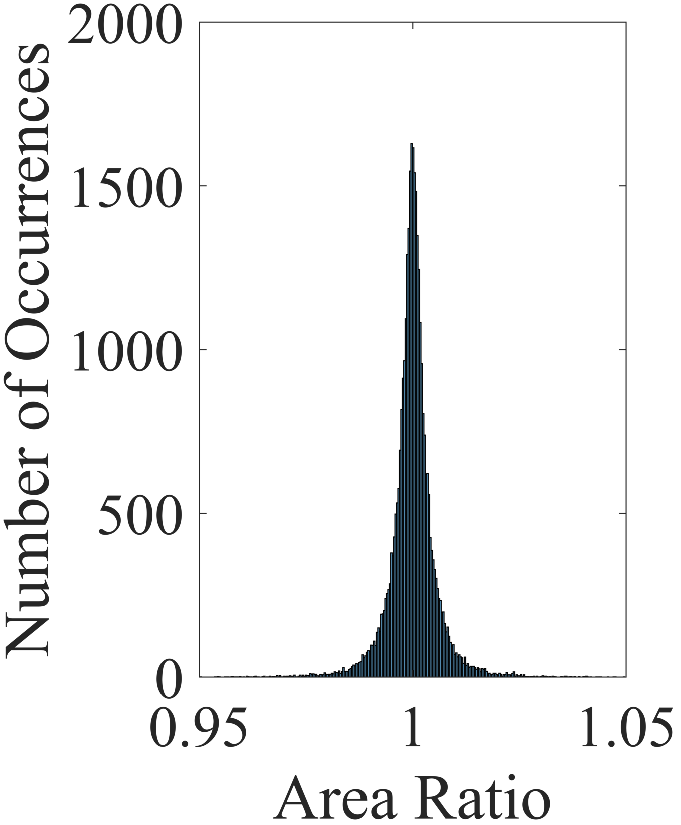} & 
\includegraphics[height=4.2cm]{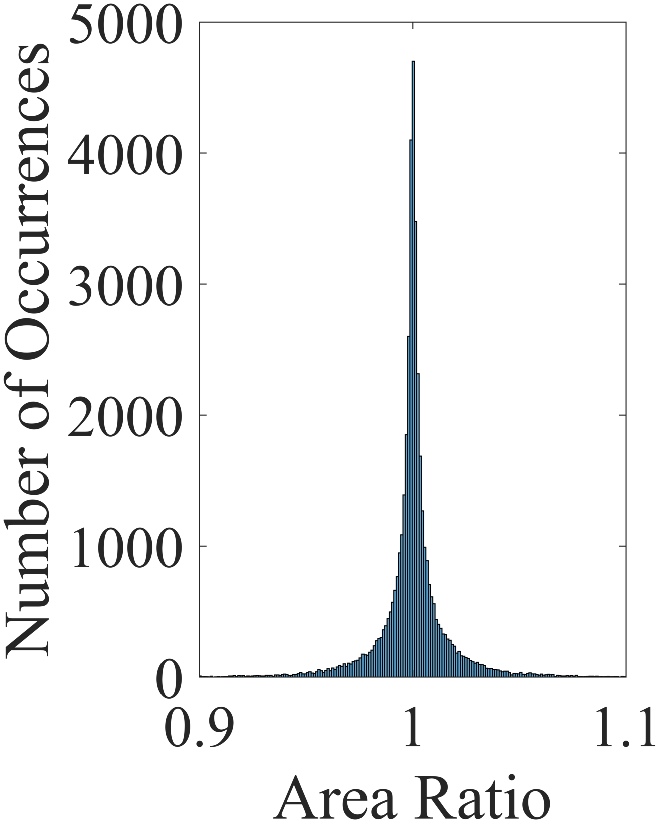} & 
\includegraphics[height=4.2cm]{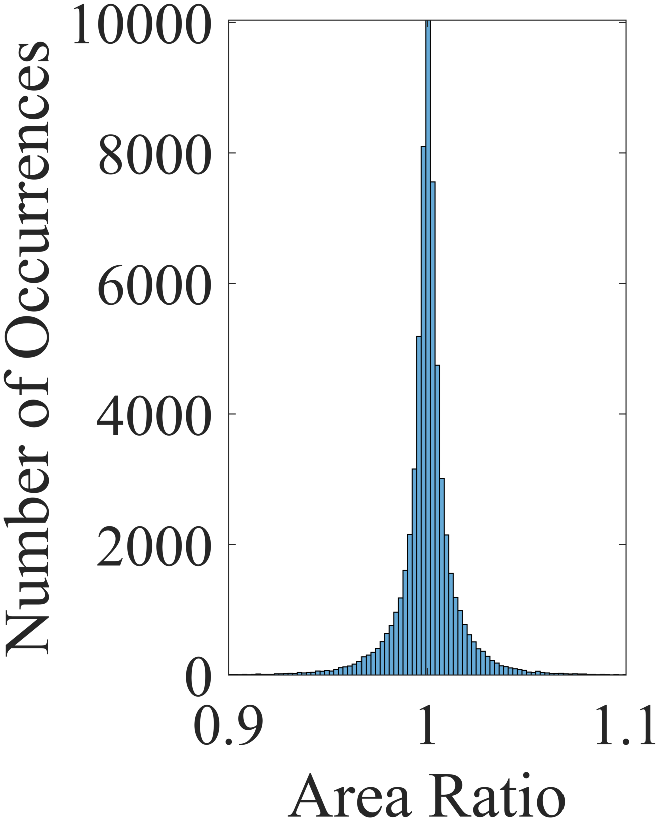}
\\[0.5cm]
Max Planck & Human Brain & Left Hand & Knit Cap Man\\
\includegraphics[height=4.2cm]{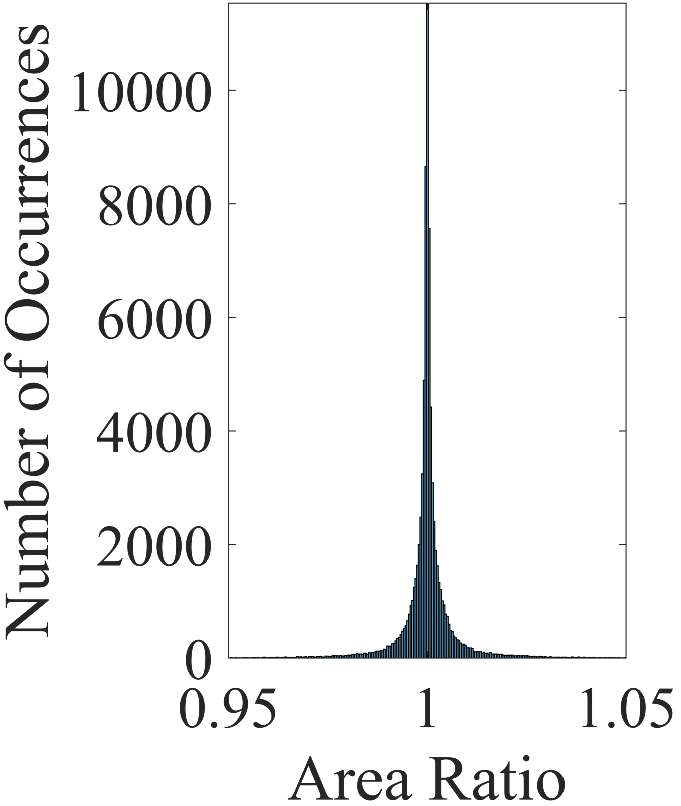} & 
\includegraphics[height=4.2cm]{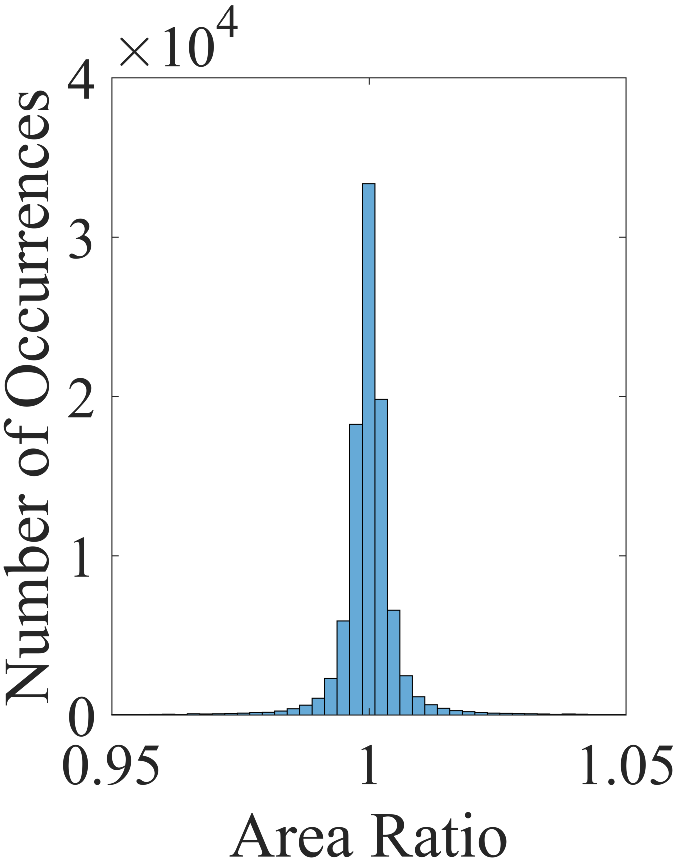} & 
\includegraphics[height=4.2cm]{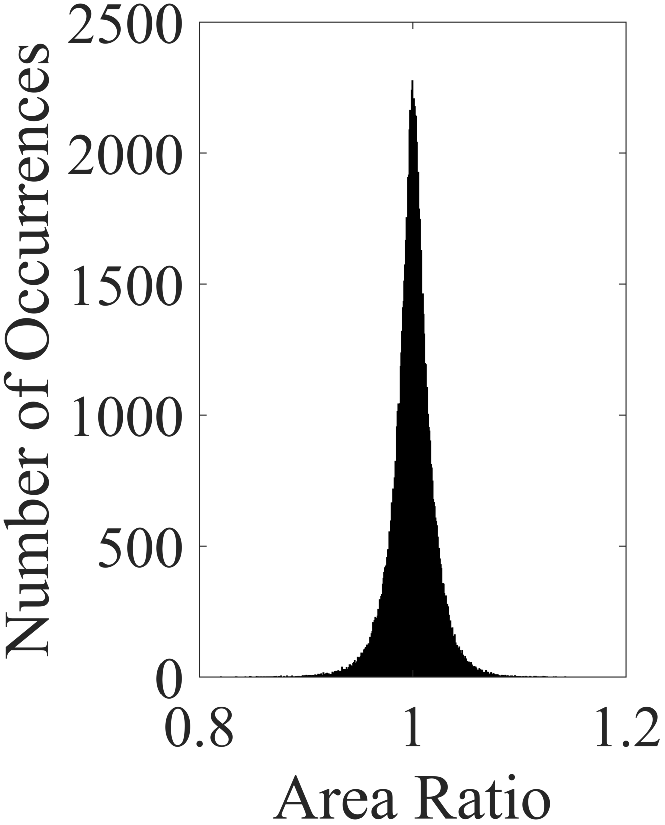} &
\includegraphics[height=4.2cm]{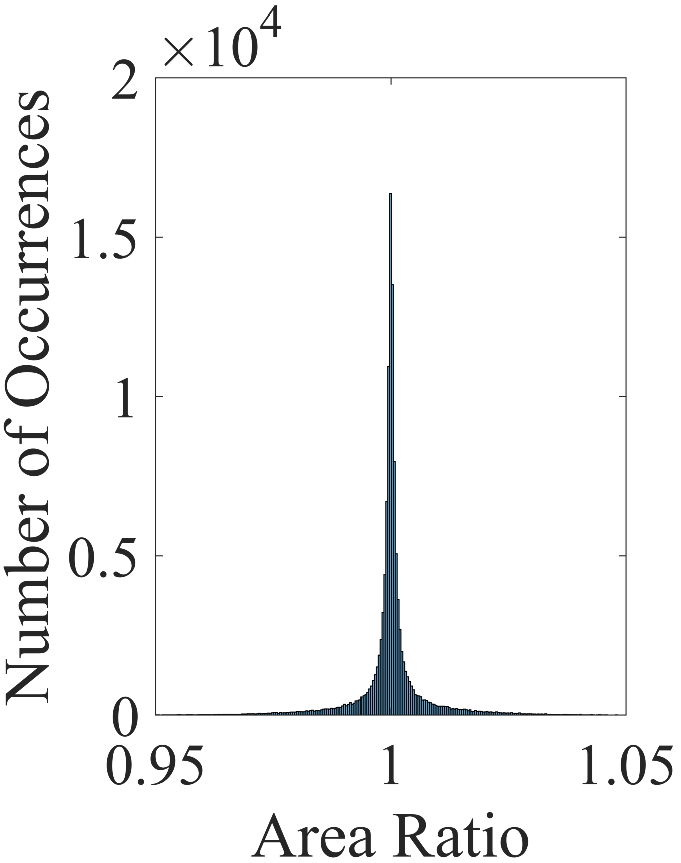} 
\\[0.5cm]
Human Face & Bimba Statue & Buddha & Nefertiti Statue \\
\includegraphics[height=4.2cm]{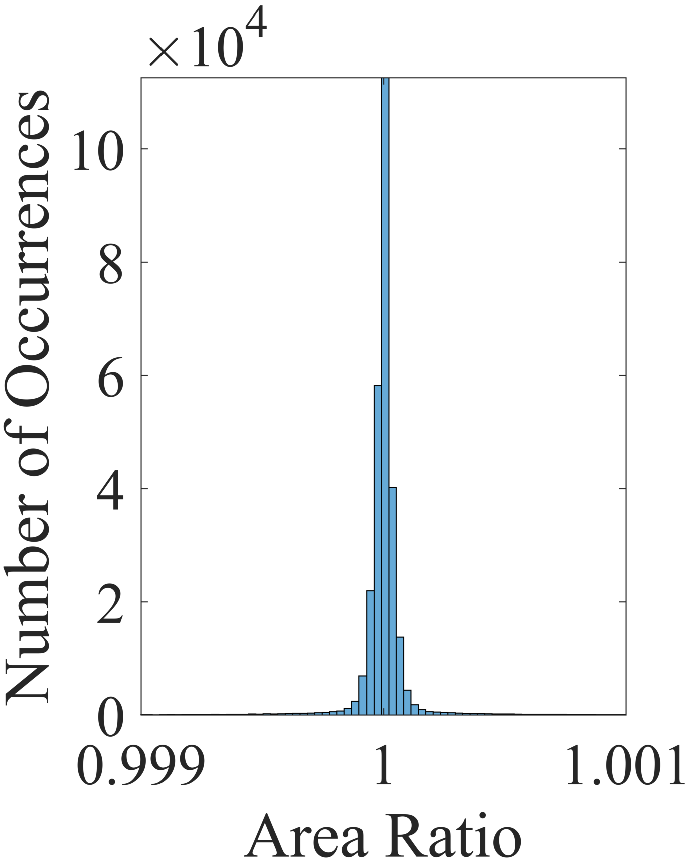} &
\includegraphics[height=4.2cm]{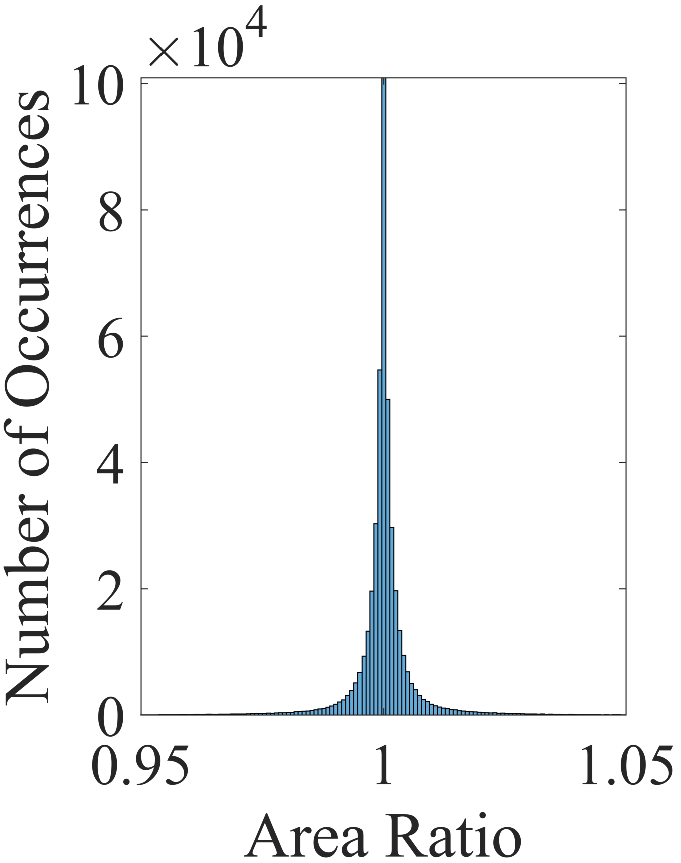} &
\includegraphics[height=4.2cm]{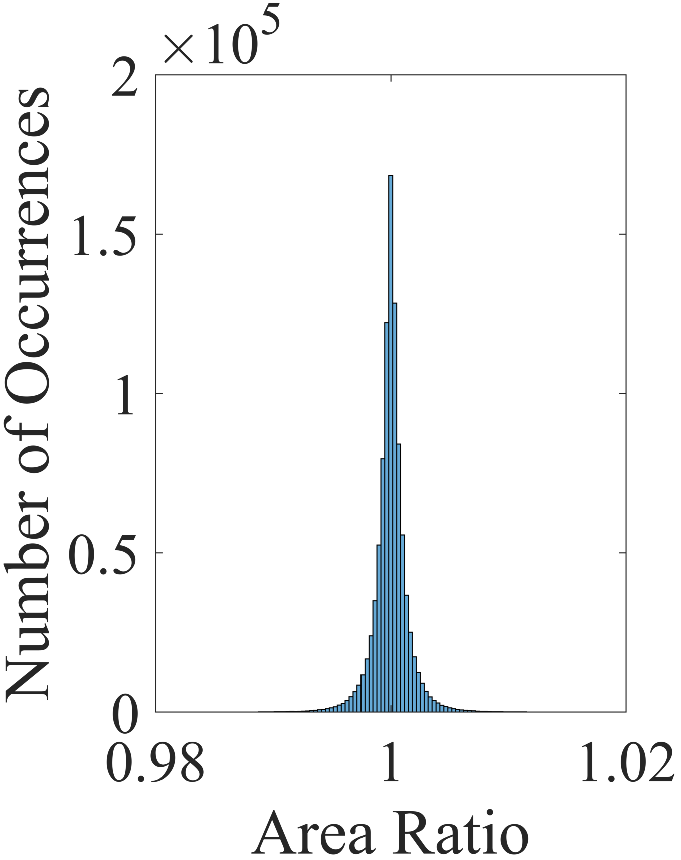} & 
\includegraphics[height=4.2cm]{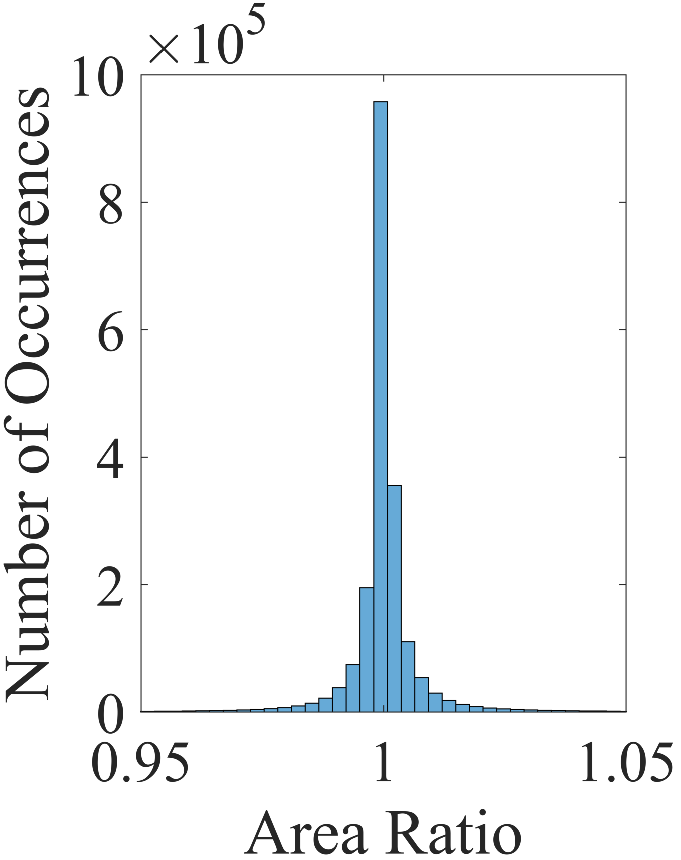}
\end{tabular}
}
\caption{Histograms of the area ratios $R_A(f,\tau)$, representing the local area distortions, for the area-preserving parameterizations of benchmark mesh models computed using the proposed nonlinear CG method, Algorithm \ref{alg:PCG}.}
\label{fig:hist}
\end{figure}

In Table \ref{tab:PCG}, we provide a comparison between the proposed nonlinear CG method of AEM, Algorithm \ref{alg:PCG}, for computing area-preserving parameterizations and a state-of-the-art fixed-point iterative algorithm of SEM \cite{YuLW19}. The number of iterations of both algorithms is set to $100$. The comparison is based on the criteria of area-preserving accuracy and computational efficiency. 
The global area distortion is measured by the normalized authalic energy defined as \eqref{eq:E_A_mod}. 
The normalized authalic energy \eqref{eq:E_A_mod} has the property that $\widetilde{E}_A(f) \geq 0$, and the equality holds if and only if $f$ is area-preserving \cite{Yueh22}. 
From Table \ref{tab:PCG}, we observe that the proposed nonlinear CG method, Algorithm \ref{alg:PCG}, exhibits notably lower levels of local and global area distortions, along with improved computational efficiency, when compared to the fixed-point method \cite{YuLW19}.

\begin{table}[]
\centering
\caption{The SD of area ratios, defined in \eqref{eq:R_A}, the normalized authalic energy $E_A$ specified in \eqref{eq:E_A_mod}, required computational time, and the number of folding triangles of area-preserving parameterizations using the fixed-point method \cite{YuLW19} and the nonlinear CG method, Algorithm \ref{alg:PCG} with $100$ iterations, respectively. 
}
\label{tab:PCG}
\resizebox{\textwidth}{!}{
\begin{tabular}{lccrcccrc}
\hline 
\multirow{3}{*}{Model name} & \multicolumn{4}{c}{Fixed-point method of SEM \cite{YuLW19}} & \multicolumn{4}{c}{Nonlinear CG method of AEM (Algorithm \ref{alg:PCG})} \\
\cline{2-9}
& Area ratios & \multirow{2}{*}{$E_A(f)$} & Time & \#Fold- & Area ratios & \multirow{2}{*}{$E_A(f)$} & Time & \#Fold- \\ 
 & {SD} & & {(secs.)} & ings & {SD} & & {(secs.)} & ings
\\
\hline 
Foot             & $4.81\times 10^{-2}$ & $7.49\times 10^{-3}$ &  2.3 & 3 & $1.47\times 10^{-2}$ & $5.08\times 10^{-4}$ &  1.7 & 0 \\ 
Chinese Lion     & $1.62\times 10^{-2}$ & $7.50\times 10^{-4}$ &  4.4 & 0 & $5.84\times 10^{-3}$ & $8.33\times 10^{-5}$ &  2.7 & 0 \\ 
Femur            & $4.41\times 10^{-2}$ & $5.92\times 10^{-3}$ &  6.6 & 0 & $1.92\times 10^{-2}$ & $1.05\times 10^{-3}$ &  4.4 & 0 \\ 
Stanford Bunny   & $3.92\times 10^{-2}$ & $4.77\times 10^{-3}$ &  8.8 & 0 & $1.66\times 10^{-2}$ & $8.30\times 10^{-4}$ &  5.1 & 0 \\ 
Max Planck       & $1.65\times 10^{-2}$ & $8.45\times 10^{-4}$ & 12.6 & 0 & $8.02\times 10^{-3}$ & $1.96\times 10^{-4}$ &  6.5 & 0 \\ 
Human Brain      & $3.52\times 10^{-2}$ & $3.78\times 10^{-3}$ & 14.4 & 0 & $1.41\times 10^{-2}$ & $5.26\times 10^{-4}$ &  6.9 & 0 \\ 
Left Hand        & $4.51\times 10^{-2}$ & $6.39\times 10^{-3}$ & 17.2 & 0 & $2.00\times 10^{-2}$ & $1.19\times 10^{-3}$ &  7.7 & 0 \\ 
Knit Cap Man     & $1.77\times 10^{-2}$ & $9.86\times 10^{-4}$ & 19.1 & 0 & $8.02\times 10^{-3}$ & $1.87\times 10^{-4}$ &  9.1 & 0 \\ 
Human Face       & $1.11\times 10^{-3}$ & $3.66\times 10^{-6}$ & 30.2 & 0 & $2.10\times 10^{-4}$ & $1.28\times 10^{-7}$ & 17.3 & 0 \\ 
Bimba Statue     & $1.52\times 10^{-2}$ & $7.09\times 10^{-4}$ & 76.0 & 0 & $7.61\times 10^{-3}$ & $1.64\times 10^{-4}$ & 35.0 & 0 \\ 
Buddha           & $3.90\times 10^{-3}$ & $4.50\times 10^{-5}$ &191.0 & 0 & $1.61\times 10^{-3}$ & $7.23\times 10^{-6}$ & 90.4 & 0 \\ 
Nefertiti Statue & $2.27\times 10^{-2}$ & $3.60\times 10^{-3}$ &336.2 & 0 & $1.33\times 10^{-2}$ & $9.33\times 10^{-4}$ &145.5 & 0 \\
\hline 
\end{tabular}
}
\end{table}

The values of SDs of area ratios and authalic energies in Table \ref{tab:PCG} are visualized in Figure \ref{fig:Comparison} (a) and (b), respectively. We observe that the maps computed by the proposed nonlinear CG method have significantly smaller values of SDs of area ratios and authalic energies than that of the fixed-point method \cite{YuLW19}, which indicates that the nonlinear CG method has better area-preserving accuracy.

\begin{figure}
\centering
\resizebox{\textwidth}{!}{
\begin{tabular}{cc}
\includegraphics[height=4.5cm]{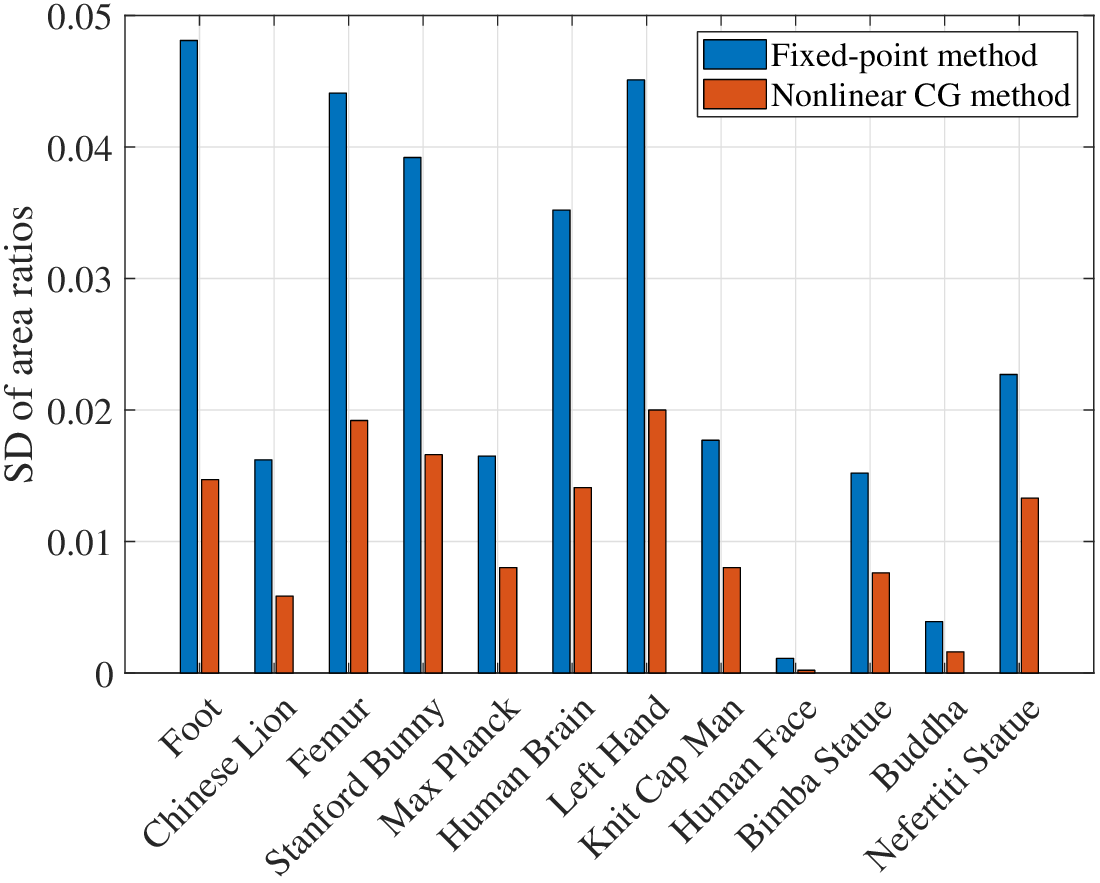} &
\includegraphics[height=4.5cm]{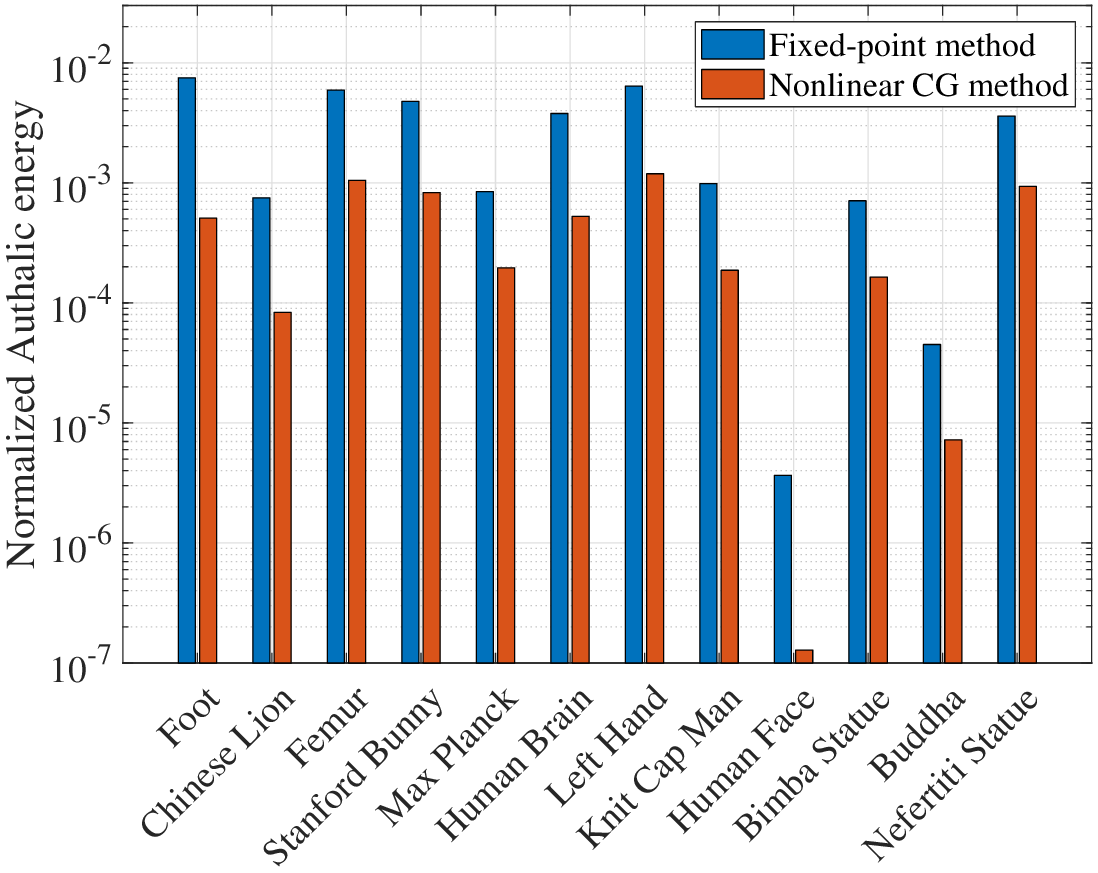} \\
(a) & (b)
\end{tabular}
}
\caption{The comparison of (a) SD of area ratios, defined in \eqref{eq:R_A}, and (b) the normalized authalic energy, defined in \eqref{eq:E_A_mod}, of the fixed-point method \cite{YuLW19} and the nonlinear CG method (Algorithm \ref{alg:PCG}).}
\label{fig:Comparison}
\end{figure}

Figure \ref{fig:Time} further shows the relationship between the numbers of vertices and the computational time cost of the fixed-point method \cite{YuLW19} and the proposed nonlinear CG method, Algorithm \ref{alg:PCG}. We observe that the efficiency of the proposed method for the AEM is significantly improved compared to the fixed-point method \cite{YuLW19}. This improvement results in a computation time of less than 150 seconds for an area-preserving parameterization of a mesh with about 1 million vertices, which is reasonably satisfactory.

Comparisons between the fixed-point method of SEM \cite{YuLW19} and the other two state-of-the-art algorithms, the stretch-minimizing method \cite{YoBS04} and the optimal mass transportation-based method \cite{SuCQ16}, can be found in \cite{YuLW19}, which indicates that the SEM \cite{YuLW19} was already the best among the three on computing disk area-preserving parameterizations.

\begin{figure}
\centering
\includegraphics[height=6cm]{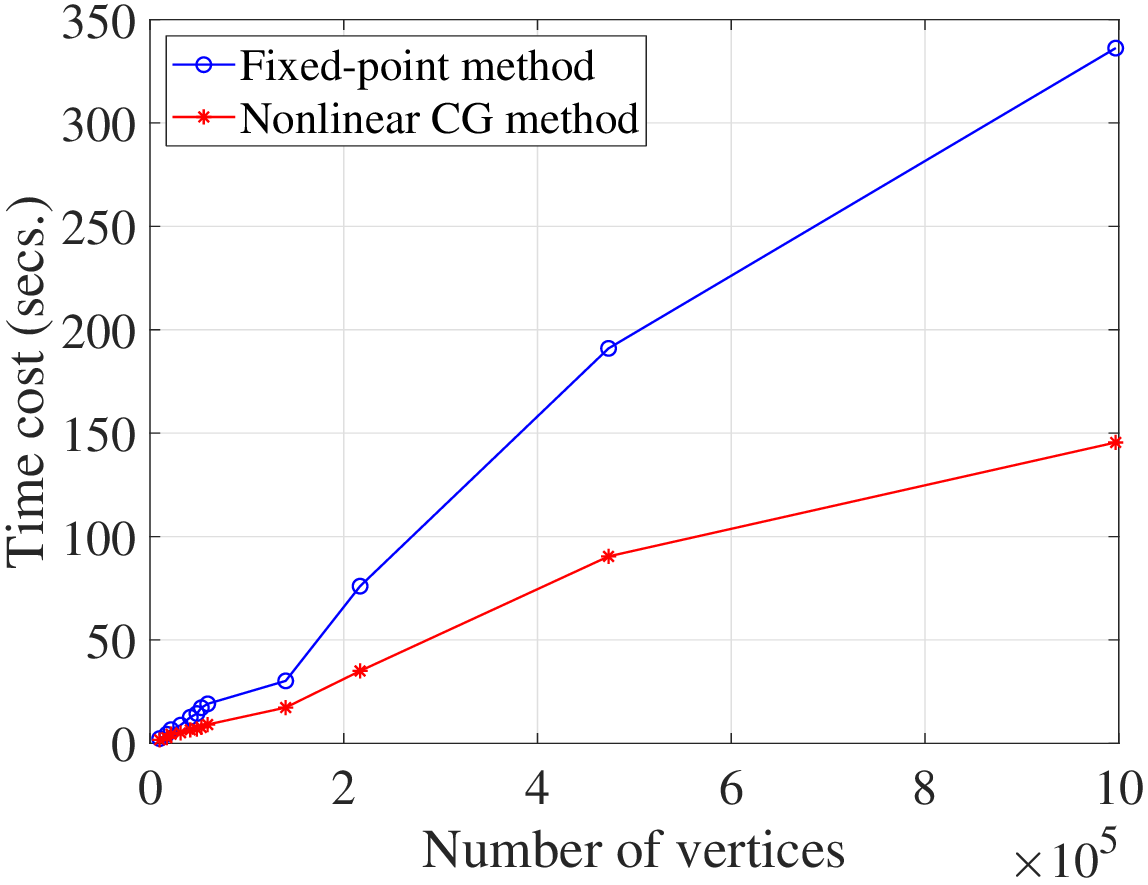}
\caption{The relationship between numbers of vertices and the computational time cost of the fixed-point method \cite{YuLW19} and the proposed nonlinear CG method (Algorithm \ref{alg:PCG})}
\label{fig:Time}
\end{figure}

The MATLAB executables with a user-friendly interface for the nonlinear CG method of the AEM (Algorithm \ref{alg:PCG}) can be found in \url{http://tiny.cc/DiskAEM}.

\section{Application to area-preserving surface registration}
\label{sec:6}

In this section, we demonstrate a straightforward application of the proposed preconditioned CG method of AEM, Algorithm \ref{alg:PCG}, for the registration task of surfaces. 
Surface registration is a widely studied problem in geometry processing \cite{LaLu14,LuLY14,YoMY14,YuLW19}, but only a few studies have addressed surface registration with area-preserving properties.

In particular, we consider the registration task of two simply connected open surfaces $\S$ and $\T$ with landmark pairs $\{(p_i,q_i) \mid p_i\in\S, \, q_i\in\T\}_{i=1}^m$. The task aims to find a bijective mapping $\varphi:\S\to\T$ that satisfies $\varphi(p_i)\approx q_i$, for $i=1, \ldots, m$. By applying the disk area-preserving parameterizations $f:\S\to\D$ and $g:\T\to\D$, the registration task is simplified into finding a bijective mapping $h:\D\to\D$ that satisfies $h\circ f(p_i) \approx g(q_i)$, for $i=1, \ldots, m$. Once we have such a map $h$, the registration mapping from $\S$ to $\T$ is constructed by $\varphi = g^{-1} \circ h \circ f$.

We suppose the triangular mesh $\S$ consists of $n$ vertices. The objective function for the registration task is formulated as a stretch energy with a penalty term given by
$$
\widetilde{E}_S(h) = E_S(h) + \lambda \sum_{i=1}^m \| h\circ f(p_i) - g(q_i) \|^2,
$$
where $\lambda$ is a positive real parameter. The gradient of $\widetilde{E}_S$ can be formulated as
\begin{align}
\nabla\widetilde{E}_S(h) &= 2 \left( (I_2\otimes L_S(h)) \, \h + \lambda (I_2\otimes P)^\top( (I_2\otimes P) \,\h - \g) \right) \nonumber \\
&= 2\left( I_2\otimes (L_S(h) +\lambda P^\top P)  \right) \h - 2\lambda (I_2\otimes P)^\top \g, \label{eq:RegGrad}
\end{align}
where $P$ is a permutation submatrix of size $m$-by-$n$ with 
\begin{equation*}
P_{i,j} =
\begin{cases}
1 & \text{if the $i$th landmark of $\S$ is $j$}, \\
0 & \text{otherwise}.
\end{cases}
\end{equation*}
The minimizer of $\widetilde{E}_S$ satisfies $\nabla\widetilde{E}_S(h) = \0$, i.e., the nonlinear system
\begin{equation} \label{eq:RegLS}
\left( I_2\otimes (L_S(h) +\lambda P^\top P)  \right) \h = \lambda (I_2\otimes P)^\top \g.
\end{equation}
Under a prescribed boundary map $\h_\B = \b$, the nonlinear system \eqref{eq:RegLS} becomes
\begin{equation} \label{eq:RegLS1}
\big( [L_S(h)]_{\I,\I} +\lambda D_{\I,\I} \big) \,\h^\ell_\I = \lambda  [P^\top \g^\ell]_\I - \big( [L_S(h)]_{\I,\B} +\lambda D_{\I,\B} \big) \,\b^\ell, ~ \ell=1,2,
\end{equation}
where $D = P^\top P$ is a diagonal matrix with
$$
D_{i,i} = 
\begin{cases}
1 & \text{ if $i$th vertices is a landmark},\\
0 & \text{ otherwise}.
\end{cases}
$$
The nonlinear system \eqref{eq:RegLS1} can be solved by the fixed-point iteration
\begin{equation} \label{eq:RegIter}
\big( [L_S(h^{(k)})]_{\I,\I} +\lambda D_{\I,\I} \big) \,{\h_\I^\ell}^{(k+1)} = \lambda [P^\top \g^\ell]_\I - \big( [L_S(h^{(k)})]_{\I,\B} +\lambda D_{\I,\B} \big) \,\b^\ell,
\end{equation}
$\ell=1,2$, with $L_S(h^{(0)}) = L_S(\mathrm{id})$. 
To effectively solve the minimization problem 
\begin{equation} \label{eq:RegProblem}
h = \argmin_{\f_\B = \b} \widetilde{E}_S(f),
\end{equation}
we apply Algorithm \ref{alg:PCG} by replacing the gradient direction with \eqref{eq:RegGrad}.

In practice, we demonstrate the registration mapping from the model Max Planck to Knit Cap Man. First, we compute area-preserving parameterizations $f$ and $g$ of Max Planck and Knit Cap Man surface models, respectively, using Algorithm \ref{alg:PCG}.
Then, we define feature curves of Max Planck and Knit Cap Man models on the disk parameter domains, as illustrated in Figure \ref{fig:Reg} (a) and (c). 
Next, we solve the minimization problem \eqref{eq:RegProblem}. 
The boundary map is chosen to be the optimal rotation with respect to the landmarks as 
$$
R = \argmin_{R\in SO(2)} \sum_{i=1}^m \| R(p_i) - q_i \|^2,
$$
which can be solved by the singular value decomposition of a 2-by-2 matrix \cite{SoAl07}. 
The initial map is computed by the fixed-point iteration \eqref{eq:RegIter} with $5$ iteration steps. 
The computed registration mapping is presented in Figure \ref{fig:Reg} (b). It is worth noting that the map $h$ is bijective with authalic energy $E_A(h) = 5.04\times 10^{-4}$, and the positions of feature curves in Figure \ref{fig:Reg} (b) and (c) are identical up to an average error of $10^{-3}$, which means the registration mapping $h$ preserves the area of the surface and keeps the specified facial features aligned well.

\begin{figure}
\centering
\resizebox{\textwidth}{!}{
\begin{tabular}{ccc}
\multicolumn{2}{c}{Max Planck} & Knit Cap Man \\
Parameterization & Registration map & Parameterization \\
\includegraphics[height=5cm]{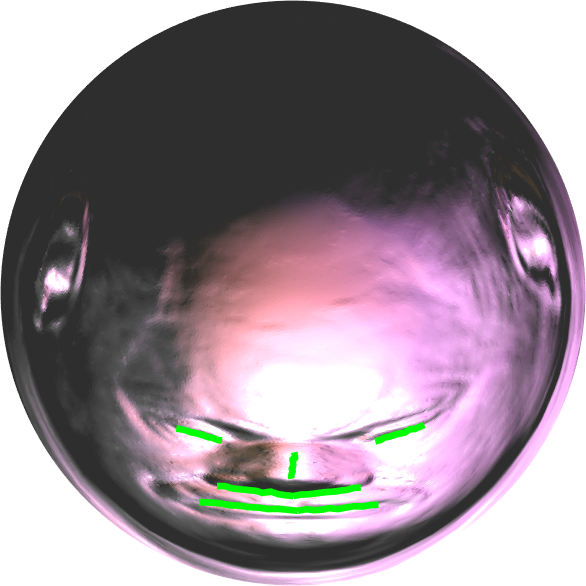} &
\includegraphics[height=5cm]{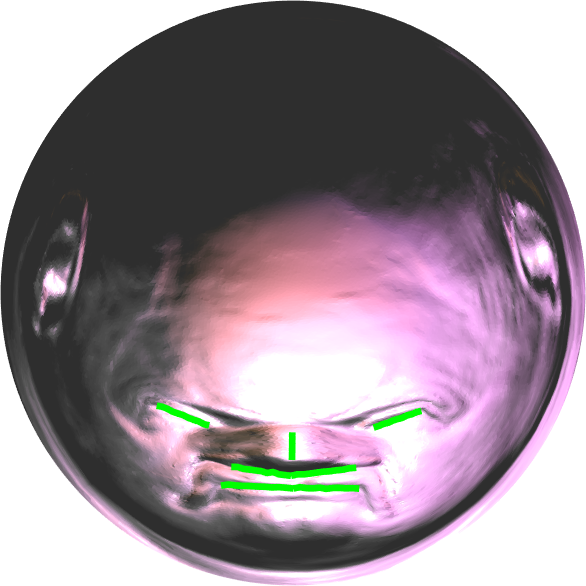} &
\includegraphics[height=5cm]{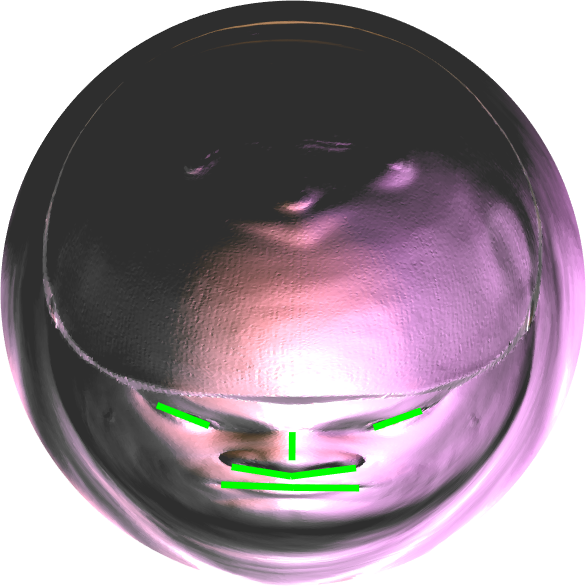} \\
(a) & (b) & (c)
\end{tabular}
}
\caption{The feature curves on the disk area-preserving parameter domains of (a) Max Planck and (c) Knit Cap Man, respectively, and (b) on the disk area-preserving registration map of Max Planck.}
\label{fig:Reg}
\end{figure}

In Figure \ref{fig:Homotopy}, we further show the linear homotopy $\mathcal{H}:\D\times[0,1]\to\R^3$ between the maps $f^{-1}$ and $g^{-1}\circ h$, defined as
\begin{equation} \label{eq:Homotopy}
\mathcal{H}(v,t) = (1-t) \, f^{-1}(v) + t \, (g^{-1}\circ h)(v).
\end{equation}
We observe that the homotopy carried out by the registration map $h$ effectively maintains the alignment of facial features of two surface models.

\begin{figure}
\centering
\resizebox{\textwidth}{!}{
\begin{tabular}{cccccc}
Max Planck &&&&& Knit Cap Man \\
\includegraphics[height=3.4cm]{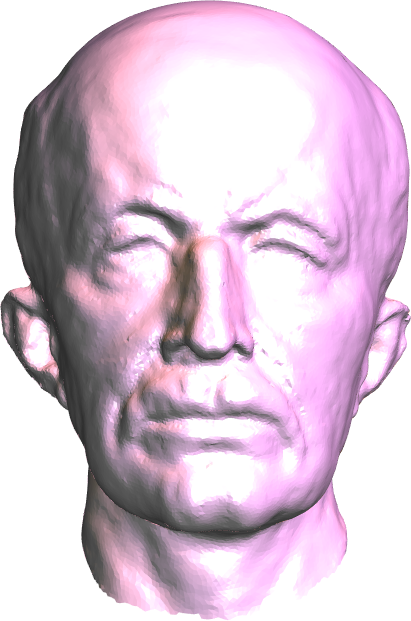} &
\includegraphics[height=3.4cm]{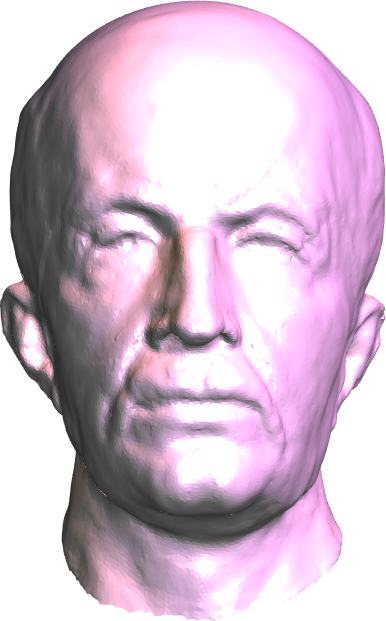} &
\includegraphics[height=3.4cm]{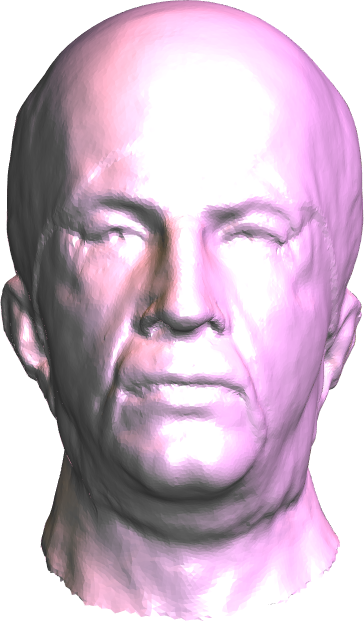} &
\includegraphics[height=3.4cm]{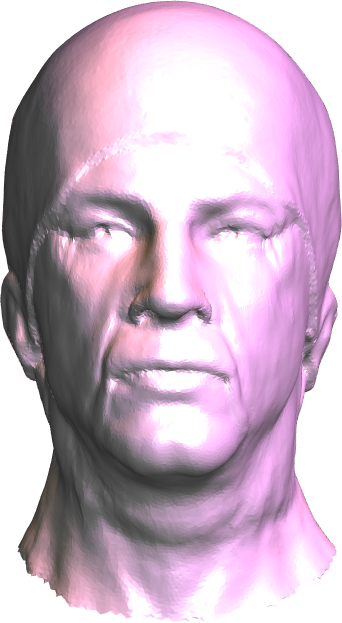} &
\includegraphics[height=3.4cm]{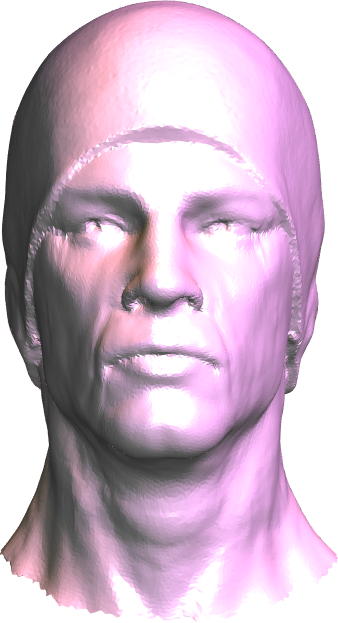} &
\includegraphics[height=3.4cm]{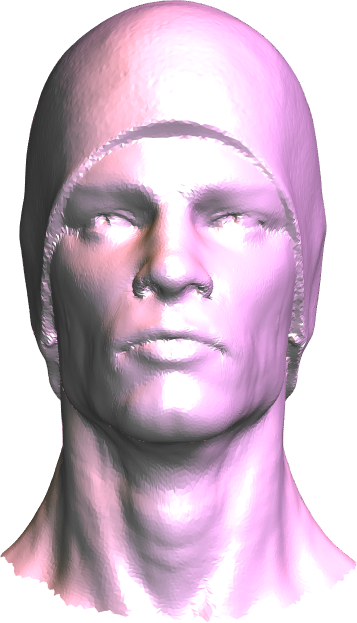} \\
$\mathcal{H}(\D,0)$ & $\mathcal{H}(\D,0.2)$ & $\mathcal{H}(\D,0.4)$ & $\mathcal{H}(\D,0.6)$ & $\mathcal{H}(\D,0.8)$ & $\mathcal{H}(\D,1)$
\end{tabular}
}
\caption{The images of the linear homotopy \eqref{eq:Homotopy} from Max Planck to Knit Cap Man.}
\label{fig:Homotopy}
\end{figure}

\section{Concluding remarks}
\label{sec:7}

We have presented a new preconditioned nonlinear CG method of AEM for the computation of area-preserving parameterizations. The effectiveness of the proposed algorithm is demonstrated to be significantly better than another state-of-the-art algorithm in terms of area-preserving accuracy and computational efficiency. In addition, we have provided rigorous proof for the convergence of the proposed algorithm. Furthermore, a straightforward application of surface registration is presented. Such encouraging results make the AEM more valuable in the practical computation of area-preserving mappings.

\section*{Acknowledgements}
This work was partially supported by the National Science and Technology Council and the National Center for Theoretical Sciences.

\bibliographystyle{abbrv}
\bibliography{references}

\begin{thebibliography}{10}

\bibitem{AIM}
{Digital Shape Workbench - Shape Repository}.
\newblock \url{http://visionair.ge.imati.cnr.it/ontologies/shapes/}.
\newblock (2016).

\bibitem{Stanford}
{The Stanford 3D Scanning Repository}.
\newblock \url{http://graphics.stanford.edu/data/3Dscanrep/}.
\newblock (2016).

\bibitem{Sketchfab}
{Sketchfab}.
\newblock \url{https://sketchfab.com/}, (2019).

\bibitem{AmDD04}
P.~R. Amestoy, T.~A. Davis, and I.~S. Duff.
\newblock Algorithm 837: {AMD}, an approximate minimum degree ordering
  algorithm.
\newblock {\em ACM Trans. Math. Softw.}, 30(3):381--388, 2004.

\bibitem{ChRy18}
G.~P.~T. Choi and C.~H. Rycroft.
\newblock Density-equalizing maps for simply connected open surfaces.
\newblock {\em SIAM J. Imag. Sci.}, 11(2):1134--1178, 2018.

\bibitem{DoTa10}
A.~Dominitz and A.~Tannenbaum.
\newblock Texture mapping via optimal mass transport.
\newblock {\em IEEE Trans. Vis. Comput. Graphics}, 16(3):419--433, 2010.

\bibitem{FlRe64}
R.~Fletcher and C.~M. Reeves.
\newblock Function minimization by conjugate gradients.
\newblock {\em The Computer Journal}, 7(2):149--154, 1964.

\bibitem{FlHo05}
M.~S. Floater and K.~Hormann.
\newblock Surface parameterization: a tutorial and survey.
\newblock In {\em Advances in Multiresolution for Geometric Modelling}, pages
  157--186. Springer Berlin Heidelberg, 2005.

\bibitem{HoLe07}
K.~Hormann, B.~L{\'e}vy, and A.~Sheffer.
\newblock Mesh parameterization: Theory and practice.
\newblock In {\em ACM SIGGRAPH Course Notes}, 2007.

\bibitem{KuLY21}
Y.-C. Kuo, W.-W. Lin, M.-H. Yueh, and S.-T. Yau.
\newblock Convergent conformal energy minimization for the computation of disk
  parameterizations.
\newblock {\em SIAM J. Imag. Sci.}, 14(4):1790--1815, 2021.

\bibitem{LaLu14}
K.~C. Lam and L.~M. Lui.
\newblock Landmark-and intensity-based registration with large deformations via
  quasi-conformal maps.
\newblock {\em SIAM J. Imag. Sci.}, 7(4):2364--2392, 2014.

\bibitem{LiJY21}
W.-W. Lin, C.~Juang, M.-H. Yueh, T.-M. Huang, T.~Li, S.~Wang, and S.-T. Yau.
\newblock {3D} brain tumor segmentation using a two-stage optimal mass
  transport algorithm.
\newblock {\em Sci. Rep.}, 11:14686, 2021.

\bibitem{LiLH22}
W.-W. Lin, J.-W. Lin, T.-M. Huang, T.~Li, M.-H. Yueh, and S.-T. Yau.
\newblock A novel 2-phase residual {U-net} algorithm combined with optimal mass
  transportation for {3D} brain tumor detection and segmentation.
\newblock {\em Sci. Rep.}, 12:6452, 2022.

\bibitem{LuLY14}
L.~M. Lui, K.~C. Lam, S.-T. Yau, and X.~Gu.
\newblock Teichmuller mapping ({T}-map) and its applications to landmark
  matching registration.
\newblock {\em SIAM J. Imag. Sci.}, 7(1):391--426, 2014.

\bibitem{NoWr06}
J.~Nocedal and S.~J. Wright.
\newblock {\em Numerical Optimization}.
\newblock Springer, New York, NY, 2006.

\bibitem{SaSG01}
P.~V. Sander, J.~Snyder, S.~J. Gortler, and H.~Hoppe.
\newblock Texture mapping progressive meshes.
\newblock In {\em Proceedings of the 28th Annual Conference on Computer
  Graphics and Interactive Techniques}, SIGGRAPH '01, pages 409--416, New York,
  NY, USA, 2001. ACM.

\bibitem{ShPR06}
A.~Sheffer, E.~Praun, and K.~Rose.
\newblock Mesh parameterization methods and their applications.
\newblock {\em Found. Trends. Comp. Graphics and Vision.}, 2(2):105--171, 2006.

\bibitem{SoAl07}
O.~Sorkine and M.~Alexa.
\newblock As-rigid-as-possible surface modeling.
\newblock In {\em Proceedings of EUROGRAPHICS/ACM SIGGRAPH Symposium on
  Geometry Processing}, pages 109--116, 2007.

\bibitem{SuCQ16}
K.~Su, L.~Cui, K.~Qian, N.~Lei, J.~Zhang, M.~Zhang, and X.~D. Gu.
\newblock Area-preserving mesh parameterization for poly-annulus surfaces based
  on optimal mass transportation.
\newblock {\em Comput. Aided Geom. D.}, 46:76 -- 91, 2016.

\bibitem{YoMY14}
Y.~Yoshiyasu, W.-C. Ma, E.~Yoshida, and F.~Kanehiro.
\newblock As-conformal-as-possible surface registration.
\newblock In {\em Comput. Graph. Forum}, volume~33, pages 257--267. Wiley
  Online Library, 2014.

\bibitem{YoBS04}
S.~Yoshizawa, A.~Belyaev, and H.~P. Seidel.
\newblock A fast and simple stretch-minimizing mesh parameterization.
\newblock In {\em Proceedings Shape Modeling Applications, 2004.}, pages
  200--208, June 2004.

\bibitem{Yueh22}
M.-H. Yueh.
\newblock Theoretical foundation of the stretch energy minimization for
  area-preserving simplicial mappings.
\newblock {\em SIAM J. Imaging Sci.}, 16(3):1142--1176, 2023.

\bibitem{YuLL19}
M.-H. Yueh, T.~Li, W.-W. Lin, and S.-T. Yau.
\newblock A novel algorithm for volume-preserving parameterizations of
  3-manifolds.
\newblock {\em SIAM J. Imag. Sci.}, 12(2):1071--1098, 2019.

\bibitem{YuLL20}
M.-H. Yueh, T.~Li, W.-W. Lin, and S.-T. Yau.
\newblock A new efficient algorithm for volume-preserving parameterizations of
  genus-one 3-manifolds.
\newblock {\em SIAM J. Imag. Sci.}, 13(3):1536--1564, 2020.

\bibitem{YuLW17}
M.-H. Yueh, W.-W. Lin, C.-T. Wu, and S.-T. Yau.
\newblock An efficient energy minimization for conformal parameterizations.
\newblock {\em J. Sci. Comput.}, 73(1):203--227, 2017.

\bibitem{YuLW19}
M.-H. Yueh, W.-W. Lin, C.-T. Wu, and S.-T. Yau.
\newblock A novel stretch energy minimization algorithm for equiareal
  parameterizations.
\newblock {\em J. Sci. Comput.}, 78(3):1353--1386, 2019.

\bibitem{ZhSG13}
X.~Zhao, Z.~Su, X.~D. Gu, A.~Kaufman, J.~Sun, J.~Gao, and F.~Luo.
\newblock Area-preservation mapping using optimal mass transport.
\newblock {\em IEEE Trans. Vis. Comput. Gr.}, 19(12):2838--2847, 2013.

\bibitem{ZoHG11}
G.~Zou, J.~Hu, X.~Gu, and J.~Hua.
\newblock Authalic parameterization of general surfaces using {L}ie advection.
\newblock {\em IEEE Trans. Vis. Comput. Graph.}, 17(12):2005--2014, 2011.

\end{thebibliography}

\end{sloppypar}
\end{document}